\documentclass[12pt]{amsart}  
\usepackage{amssymb}
\usepackage{hyperref}
\usepackage{comment}
\newtheorem{theorem}{Theorem}
\newtheorem{lemma}[theorem]{Lemma}
\newtheorem{proposition}[theorem]{Proposition}

\newtheorem{corollary}[theorem]{Corollary}
\newtheorem{remar}[theorem]{Remark}
\newenvironment{remark}{\begin{remar}\rm}{\end{remar}}
\newcommand{\bfind}[1]{\index{#1}{\bf #1}}
\newcommand{\n}{\par\noindent}

\newcommand{\sn}{\par\smallskip\noindent}
\newcommand{\mn}{\par\medskip\noindent}
\newcommand{\bn}{\par\bigskip\noindent}
\newcommand{\pars}{\par\smallskip}
\newcommand{\parm}{\par\medskip}
\newcommand{\parb}{\par\bigskip}

\newcommand{\sep}{^{\rm sep}}

\newcommand{\chara}{\mbox{\rm char}\,}

\newcommand{\ic}{\mbox{\rm IC}\,}
\newcommand{\kras}{\mbox{\rm kras}}
\newcommand{\Kras}{\mbox{\rm Kras}}
\newcommand{\Gal}{\mbox{\rm Gal}\,}
\newcommand{\rme}{\mbox{\rm e}\,}
\newcommand{\rmf}{\mbox{\rm f}\,}

\newcommand{\eu}{\mathfrak}

\newcommand{\N}{\mathbb N}

\newcommand{\cal}{\mathcal}

\begin{document}
\title{Tame Key polynomials}


\author{Arpan Dutta and Franz-Viktor Kuhlmann}

\address{Department of Mathematics, IISER Mohali, Knowledge City, Sector 81, Manauli PO,
SAS Nagar, Punjab, India, 140306.}
\email{arpan.cmi@gmail.com}

\address{Institute of Mathematics, ul.~Wielkopolska 15, 70-451 Szczecin, Poland}
\email{fvk@usz.edu.pl}

\date{12.\ 7.\ 2022}

\thanks{The first author is supported by a Post-Doctoral Fellowship of the National Board of 
Higher Mathematics, India. The second author is supported by Opus grant 2017/25/B/ST1/01815 
from the National Science Centre of Poland.}


\begin{abstract}
We introduce a new method of constructing complete sequences of key polynomials for simple
extensions of tame fields. In our approach the key polynomials are taken to be the minimal 
polynomials over the base field of suitably constructed elements in its algebraic closure,
with the extensions generated by them forming an
increasing chain. In the case of algebraic extensions, we generalize the results to 
countably generated infinite tame extensions over henselian but not necessarily tame fields.
In the case of transcendental extensions, we demonstrate the central role that is played by
the implicit constant fields, which reveals the tight connection with the algebraic case.
\end{abstract}

\maketitle
%
%
%
\section{Introduction}
In this paper we will work with (Krull) valuations on fields and their extensions to
rational function fields. Note that we {\it always identify equivalent valuations}.
For basic information on valued fields and for notation, see Section~\ref{sectprel}.
The value group $v(L^\times)$ of a valued field $(L,v)$ will be denoted by $vL$, and its
residue field by $Lv$. The value of an element $a$ 
will be denoted by $va$, and its residue by $av$.

Take a valued field $(K,v)$. It is an important task to describe, analyze and 
classify all extensions of the valuation $v$ from $K$ to the rational function
field $K(x)$. In order to be able to compute the value of every element of $K(x)$ 
with respect to $v$, it suffices to be able to compute
the value of all polynomials in $x$, that is, we only have to deal with the polynomial
ring $K[x]$. Indeed, if $f,g\in K[x]$, then necessarily, $v\frac f g =vf-vg$. We know the
values of all elements in $K$. If in addition we know the 
value $vx$, then everything would be easy if for every polynomial
\begin{equation}                                   \label{pol}
f(x)\>=\>\sum_{i=0}^n c_ix^i\in K[x]
\end{equation}
the following equation would hold:
\begin{equation}                                            \label{minval}
vf(x)\>=\> \min_{0\leq i\leq n} vc_i + ivx\>.
\end{equation}
We can define valuations on $K(x)$ in this way by choosing $vx$ to be any 
element in some ordered abelian group which contains $vK$. If we choose $vx=0$, we 
obtain the \bfind{Gau{\ss} valuation}.

But what if Equation~(\ref{minval}) does not always hold? Then there are 
{\it polynomials of unexpected value}, the value of which is larger than the minimum 
of the values of its monomials. This observation has led to the theory of {\it key
polynomials}, on which by now an abundant number of articles are available. Several of 
them present {\it complete sequences of key polynomials} in order to describe, or 
construct, all extensions.

In the present paper, we add a new aspect. When working over tame fields $(K,v)$,
we are able to prove stronger results than those for the case of general valued base 
fields. An algebraic extension $(L|K,v)$ is called \bfind{tame} if it is 
\bfind{unibranched}, i.e., the extension of $v$ from $K$ to $L$ is unique, 
and every finite subextension $E|K$ of $L|K$ satisfies the following conditions:
\sn
(TE1) the ramification index $(vE:vK)$ is not divisible by $\chara Kv$.
\n
(TE2) the residue field extension $Ev|Kv$ is separable.
\n
(TE3) the extension $(E|K,v)$ is \bfind{defectless}, i.e.,
\[
[E:K]\>=\>(vE:vK)[Ev:Kv]\>.
\]
A henselian field $(K,v)$ is called a \bfind{tame field} if its algebraic closure 
$\widetilde{K}$ with 
the unique extension of the valuation is a tame extension. The Lemma of Ostrowski
(see \cite{En,[R1]}) shows that every henselian valued field of residue characteristic $0$ 
is a tame field.

For the formulation of our main theorems we will need some more definitions. For an 
arbitray extension $(K(x)|K,v)$, we set 
\[
v(x-K)\>:=\> \{v(x-c)\mid c\in K\}\>. 
\]

A transcendental extension $(K(x)|K,v)$ is \bfind{valuation algebraic} if 
$vK(x)/vK$ is a torsion group and the residue field extension $K(x)v|Kv$ is algebraic;
otherwise, it is called \bfind{valuation transcendental}. In \cite{KTrans} we introduced
the notion of an extension $(K(x)|K,v)$ being weakly pure; here we will give a simpler, but
equivalent, definition (for the equivalence, see Section~\ref{sectwp}). We say that the 
extension $(K(x)|K,v)$ is \bfind{weakly pure} ({\bf in} $x$) if there is $a\in K$ 
such that $v(x-a)=\max v(x-\widetilde{K})$, or $x$ is limit of a pseudo Cauchy sequence in 
$(K,v)$ of transcendental type. For background on pseudo Cauchy sequences, see \cite{Ka}.

\pars
Given an arbitrary simple extension $(K(x)|K,v)$ and a polynomial $f\in K[X]$, then 
we define
\begin{equation}                               \label{deltax}   
\delta(f)\>:=\> \max\{ v(x-a)\mid a \mbox{ is a root of } f \}\>. 
\end{equation}
A root $a$ of $f$ such that $\delta(f) = v(x-a)$ is said to be a \bfind{maximal root of $f$}. 
A monic polynomial $Q(X)\in K[X]$ is said to be a \bfind{key polynomial for $(K(x)|K,v)$} if 
\[ 
\deg f < \deg Q \>\Rightarrow\> \delta(f) < \delta(Q) \mbox{ for all } f\in K[X]\>. 
\]
Further, a sequence $(Q_i)_{i\in S}$ of key polynomials is said to be \bfind{complete} if
for every $f\in K[X]$ there exists $i\in S$ such that $\deg Q_i \leq \deg f$ and 
$\delta(f) \leq \delta(Q_i)$. 

\pars
Given any extension $(F|K,v)$, we take $\ic(F|K,v)$ to be the relative algebraic closure 
of $K$ in a fixed henselization of $(F,v)$ and call it the \bfind{implicit constant
field} of $(F|K,v)$. Since the henselization $F^h$ of any valued field $(F,v)$ is unique 
up to valuation preserving isomorphism over $F$, the implicit constant field
is unique up to valuation preserving isomorphism over~$K$. In \cite{KTrans} the notion of 
``weakly pure'' is instrumental in constructing an extension of $v$ from $K$ to the rational
function field $K(x)$ such that $\ic(F|K,v)$ is equal to any given countably generated 
separable-algebraic extension of $K$.

\pars
The following is our main theorem for the case of simple transcendental extensions:
\begin{theorem}                         \label{MT1}
Let $(K,v)$ be a tame valued field and $(K(x)|K,v)$ a transcendental extension. 
Then there exist monic irreducible polynomials $\{Q_i\}_{i\in S}$ 
and $\{Q_\nu\}_{\nu\in\Omega}$ over 
$K$, where $S$ is an initial segment of $\N$, $\Omega=\emptyset$ or $\Omega= 
\{ \nu|\nu<\lambda \}$ for some limit ordinal $\lambda$, having the following 
properties:
\sn 
(i) \ $\{ Q_i \}_{i\in S} \cup \{Q_\nu\}_{\nu\in\Omega}$ forms a complete sequence of key 
polynomials for $(K(x)|K,v)$,
\sn 
(ii) \ $\deg Q_1=1$,
\sn 
(iii) \ there exist unique maximal roots $a_i$ of $Q_i$ and $z_\nu$ of $Q_\nu\,$, 
\sn
(iv) \ $K(a_{i-1}) \subsetneq K(a_i)$ and $\deg Q_{i-1} < \deg Q_i$ for $1<i\in S$,
\sn
(v) \ $v(x-a_i)>v(x-a_{i-1})$ for $1<i\in S$,
\sn
(vi) \ $v(x-a_i)=\max v(x-K(a_i))$ if $i\in S$ is not its last element or 
$\Omega=\emptyset$,
\sn
(vii) \ $v(x-a_n)=\max v(x-\widetilde{K})$ if $S= \{1,\dotsc,n\}$ and $\Omega=\emptyset$,
\sn
(viii) \ if $\Omega\neq\emptyset$, then $S$ is finite, and if $n$ is its last element, then 
$(z_\nu)_{\nu<\lambda}$ is a pseudo Cauchy sequence of transcendental type in 
$(K(a_n),v)$ and we have that $\deg Q_n = \deg Q_\nu$ for all $\nu\in\Omega$,
\sn
(ix) \ \ $\ic(K(x)|K,v)=K(a_i\mid i\in S)$, which is equal to $K(a_n)$ if 
$S= \{1,\dotsc,n\}$,
\sn
(x) \ for each $k\in S$, $\{ Q_i\}_{ 1\leq i\leq k}$ forms a complete sequence of key 
polynomials for $(K(a_k)|K,v)$.
\pars
With $L:= \ic(K(x)|K,v)$, the extension $(L(x)|L,v)$ is weakly pure.
The extension $(K(x)|K,v)$ is valuation algebraic if and only if
$S=\N$ or $\Omega\neq \emptyset$. In both cases, the extension $(L(x)|L,v)$ is immediate.
In the case of $S=\N$, $(a_i)_{i\in\N}$ is a pseudo Cauchy sequence in $(L,v)$ of
transcendental type with $x$ as its limit.
\end{theorem}

From assertion (ix) we conclude:
\begin{corollary}                         \label{MT1cor}
Let $(K,v)$ be a tame field and $(K(x)|K,v)$ a transcendental extension. Then
$\ic(K(x)|K,v)$ is a countably generated separable-algebraic extension of $K$.
\end{corollary}

Instead of Theorem~\ref{MT1}, we will prove the following generalization:
\begin{theorem}                         \label{MT1gen}
Let $(K,v)$ be a henselian valued field and $(K(x)|K,v)$ a transcendental extension.  
Assume that there is a tame extension $(L'|K,v)$ such that for some extension of $v$
from $K(x)$ to $L'(x)$, the extension $(L'(x)|L',v)$ is weakly pure. Then the assertions 
of Theorem~\ref{MT1} hold, and $\ic(K(x)|K,v)\subseteq L'$.
\end{theorem}
This theorem indeed implies Theorem~\ref{MT1} since if $(K,v)$ is a tame field, then the
extension $(\widetilde{K}|K,v)$ is tame, and by \cite[Proposition 5.2]{K66} the
extension $(\widetilde{K}(x)|\widetilde{K},v)$ is always weakly pure. 

We will construct the key polynomials $\{Q_i\}_{i\in S}$ by first constructing the 
sequence $\{a_i\}_{i\in S}$ and then taking $Q_i$ to be the minimal polynomial of $a_i$ 
over $K$. For this purpose, we revisit the notion of {\it homogeneous sequences} that was
introduced in \cite{KTrans} (see also \cite{BK}). We develop a stronger version, which we 
call {\it key sequences}, in Section~\ref{sectks}.
This will be done simultaneously for transcendental and algebraic simple extensions.
The latter leads to our main theorem for the algebraic case:

\begin{theorem}                      \label{MT2}
Let $(K,v)$ be a henselian valued field and $(K(x)|K,v)$ a tame algebraic
extension. Then there exist monic irreducible polynomials $\{Q_i\}_{i\in S}$ over 
$K$, where $S=\{1,\ldots,n\}$ for some $n\in\N$, having the following properties:
\sn 
(i) \ $\{Q_i\}_{i\in S}$ forms a complete sequence of key polynomials for $(K(x)|K,v)$,
\sn 
(ii) \ $\deg Q_1=1$, 
\sn 
(iii) \ there exist unique maximal roots $a_i$ of $Q_i\,$, 
\sn
(iv) \ $K(a_{i-1}) \subsetneq K(a_i)$ and $\deg Q_{i-1} < \deg Q_i$ for $1<i\in S$,
\sn
(v) \ $v(x-a_i)>v(x-a_{i-1})$ for $1<i\in S$,
\sn
(vi) \ $v(x-a_i)=\max v(x-K(a_i))$ for all $i\in S$,
\sn
(vii) \ $a_n=x$.
\sn
Further, also assertion (x) of Theorem~\ref{MT1} holds.
\end{theorem}

\parm
In the case of a transcendental extension $(K(x)|K,v)$, the implicit constant field 
$L:= \ic(K(x)|K,v)$ can be 
an infinite extension of $K$ (cf.\ \cite[Proposition 3.16]{KTrans}). In the situation of 
Theorem~\ref{MT1}, it is generated by the elements $a_i\,$, $i\in\N$, and assertion (x) 
shows that for every $k\in\N$, $\{ Q_i\}_{ 1\leq i\leq k}$ forms a complete sequence of 
key polynomials for $(K(a_k)|K,v)$. So we are tempted to state that 
$\{ Q_i\}_{i\in S}$ forms a complete sequence of key polynomials for $(L|K,v)$.
However, so far our definition of key polynomials does not cover cases of extensions that
are not simple. In order to address the case of implicit constant fields that are infinite 
extensions of the base field, and more generally, countably generated algebraic extensions,
we generalize our definition in the following way. For the value defined in (\ref{deltax}) 
we will now write $\delta_x(f)$. A sequence $(Q_i)_{i\in S}$ of monic irreducible 
polynomials $Q_i(X)\in K[X]$, where $S$ is as before, will be said to be a \bfind{strongly
complete sequence of key polynomials for $(L|K,v)$} if the sequence $(\deg Q_i)_{i\in S}$ 
is strictly increasing and there are roots $a_i$ of $Q_i$ such 
that the following conditions hold:
\sn
(SCKP1) \ $L=K(a_i\mid i\in S)$, 
\sn
(SCKP2) \ for all $k\in S$, $(Q_i)_{i\leq k}$ is a complete sequence of key polynomials for 
$(K(a_k)|K,v)$.
\sn
It is not a priori clear that in the settings we considered so far, every complete sequence 
of key polynomials for the corresponding algebraic extensions is also strongly complete.
However, assertion (x) of Theorems~\ref{MT1} and~\ref{MT2} implies:
\begin{proposition}                                \label{SCS1}
1) If $\{Q_i\}_{i\in S}$ is a sequence of key polynomials for $(K(x)|K,v)$, then for every
$k\in S$, $\{Q_i\}_{1\leq i\leq k}$ is a strongly complete sequence of key polynomials for 
$(K(a_k)|K,v)$.
\sn
2) In the setting of Theorems~\ref{MT1} and~\ref{MT1gen}, with $L=\ic(K(x)|K,v)$, we have 
that $(Q_i)_{i\in S}$ is a strongly complete sequence of key polynomials for $(L|K,v)$. 
\sn
3) In the setting of Theorem~\ref{MT2}, $(Q_i)_{i\in S}$ is a strongly complete sequence 
of key polynomials for $(K(x)|K,v)$.
\end{proposition}

The following is a generalization of Theorem~\ref{MT2} to the case of infinite algebraic
extensions:
\begin{theorem}                               \label{SCS2}
Take a henselian field $(K,v)$.
For every countably generated tame extension $(L|K,v)$ there exists a strongly complete
sequence of key polynomials $(Q_i)_{i\in S}$ such $\deg Q_1=1$ and also the following holds:
if $a_i\,$, $i\in S$, are the roots of the polynomials $Q_i$ that satisfy (SCKP1) and 
(SCKP2), then the following hold:
\sn
1) assertions (iv) and (x) of Theorem~\ref{MT2}, 
\sn
2) assertions (v) and (vi) of Theorem~\ref{MT2}, with $a_j$ in place of $x$ for any 
$j\in S$, $j>i$, 
\sn
3) each $a_i$ is the unique maximal root of $Q_i$ in the following sense: if $i\leq k\in S$
and $a'_i\ne a_i$ is another root of $Q_i$, then $\delta_{a_k}(Q_i)=v(a_k-a_i)>v(a_k-a'_i)$.
\end{theorem}

\bn
%
%
%
\section{Preliminaries}                        \label{sectprel}
%

%
%
%
\subsection{Tame and defectless extensions and fields, and the set 
$v(x-K)$}     \label{sectpreltame}
\mbox{ }\sn
A valued field is called \bfind{algebraically maximal} if it does not admit nontrivial 
immediate algebraic extensions.
\begin{lemma}                         \label{tameprop}
Take a valued field $(K,v)$ and any extension of $v$ to $\widetilde{K}$.
\sn
1) Every algebraic extension of a tame field is again a tame field.
\sn
2) A unibranched extension $(L|K,v)$ is tame if and only if $(L^h|K^h,v)$ is tame.
\sn
3) If $(K,v)$ is henselian and $(K_i,v)$, $i\in I$, are tame extensions of $(K,v)$, then 
so is their compositum, i.e., the smallest extension that contains all $K_i\,$.
\sn
4) Assume that $(L_1|L,v)$ and $(L_2|L_1,v)$ are algebraic 
extensions. Then $(L_2|L,v)$ is a tame extension if 
and only if $(L_2|L_1,v)$ and $(L_1|L,v)$ are tame extensions.
\sn
5) Tame fields are algebraically maximal.
\end{lemma}
\begin{proof}
1): \ This follows from \cite[part (b) of Lemma 2.17]{[K7]}.
\sn
2): \ Since the extension $(L|K,v)$ is unibranched, $L|K$ is linearly disjoint from the
extension $K^h|K$ by \cite[Lemma 2.1]{BlKu3}, and the same holds for every subextension
$E|K$. Since $E^h=E.K^h$ and henselizations are immediate extension, (TE1) and (TE2) hold 
for $E$ and $K$ if and only they hold for $E^h$ and $K^h$ in place of $E$ and $K$,
respectively. Since $[E:K]=[E.K^h:K^h]=[E^h:K^h]$, the same is true for (TE3).
\sn
3): \ By \cite[part (b) of Lemma 2.13]{[K7]}, the absolute ramification field $K^r$ of
a henselian field $(K,v)$ is its unique maximal tame extension. Hence if $(K_i,v)$, 
$i\in I$, are tame extensions of $(K,v)$, then they are all contained in $K^r$ and so is 
their compositum, which consequently is also a tame extension of $(K,v)$.

\sn
4): \ For henselian fields, this is \cite[part (a) of Lemma 2.13]{[K7]}. In the general 
case, $(L_2|L,v)$ is unibranched if and only if $(L_2|L_1,v)$ and $(L_1|L,v)$ are, so we 
can use part 2) of our lemma to reduce to the henselian case.
\sn
5): \ This follows from \cite[Theorem 3.2]{[K7]}.
\end{proof}

\begin{lemma}                        \label{nomax}
Take any extension $(K(x)|K,v)$.
\sn
1) If the extension $(K(x)|K,v)$ is immediate, then $v(x-K)$ has no largest element.
\sn
2) If $v(x-K)$ has a largest element and $v(x-y)>v(x-K)$, then 
the extension $(K(y)|K,v)$ is not immediate.
\sn
3) If $v(x-K)$ has no largest element, then $x$ is the 
limit of a pseudo Cauchy sequence in $(K,v)$ without a limit in $K$. 
\end{lemma}
\begin{proof}
1): \ This is \cite[part 2) of Lemma 2.9]{K66}. 
\sn
2): \ By \cite[part 4) of Lemma 2.9]{K66}, $v(y-K)=v(x-K)$, hence also $v(y-K)$ has no
largest element. Now part 1)of our lemma shows that $(K(y)|K,v)$ is not immediate.
\sn
3): \ The proof is a straightforward adaptation of the proof of \cite[Theorem~1]{Ka}.
\end{proof}

\begin{lemma}                               \label{limtpcs}
Suppose that in some valued field extension of $(K,v)$,
$x$ is the pseudo limit of a pseudo Cauchy sequence in $(K,v)$ of
transcendental type. Then $(K(x)|K,v)$ is immediate and $x$ is
transcendental over~$K$.
\end{lemma}
\begin{proof}
Assume that $(a_\nu)_{\nu<\lambda}$ is a pseudo Cauchy sequence in
$(K,v)$ of transcendental type. Then by \cite[Theorem~2]{Ka} there is an
immediate extension $w$ of $v$ to the rational function field $K(y)$
such that $y$ becomes a pseudo limit of $(a_\nu)_{\nu<\lambda}\,$;
moreover, if also $x$ is a pseudo limit of $(a_\nu)_{\nu<\lambda}$ in
$(K(x),v)$, then $x\mapsto y$ induces a valuation preserving isomorphism
from $K(x)$ onto $K(y)$ over $K$. Hence, $(K(x)|K,v)$ is immediate and
$x$ is transcendental over~$K$.
\end{proof}

\begin{lemma}                         \label{maxappr}
1) If $(K(b)|K,v)$ is a unibranched defectless algebraic extension, then $v(b-K)$ has a 
maximal element. 
\sn
2) Every pseudo Cauchy sequence in an algebraically maximal field without a limit in that 
field is of transcendental type.
\sn
3) Assume that $(K,v)$ an algebraically maximal field, $(K(x)|K,v)$ is an immediate 
extension, and $v(x-K)$ has no largest element. Then $x$ is the limit of a pseudo Cauchy
sequence of transcendental type in $(K,v)$.
\end{lemma}
\begin{proof}
1): \ This is \cite[part a) of Lemma 7]{BK54}.
\sn
2): \ By \cite[Theorem~3]{Ka}, a pseudo Cauchy sequence of algebraic in a field $(K,v)$ 
without a limit in $K$ gives rise to a nontrivial immediate algebraic extension of $(K,v)$,
so $(K,v)$ cannot be algebraically maximal.
\sn
3): \ This follows from part 2) together with part 3) of Lemma~\ref{nomax}.
\end{proof}

Since tame extensions are unibranched and defectless by definition, and tame fields are 
algebraically maximal by part 5) of Lemma~\ref{tameprop}, we obtain:
\begin{corollary}                        \label{maxapprcor}
1) If $(K(b)|K,v)$ is a tame extension, then $v(b-K)$ has a maximal element.   
\sn
2) Every pseudo Cauchy sequence in a tame field without a limit in that 
field is of transcendental type.
\end{corollary}

\mn
%
%
%
\subsection{Weakly pure extensions}     \label{sectwp}
\mbox{ }\sn
The following shows that our definition of ``weakly pure extension'' given in the 
Introduction coincides with the definition given in \cite{KTrans}:
\begin{lemma}                         \label{WP}
Take a valued field $(K,v)$ and an extension of $v$ from $K$ to the rational function field
$K(x)$. Then for $a\in K$, the following are equivalent:
\sn
a) $v(x-a)=\max v(x-\widetilde{K})$,
\sn
b) $v(x-a)$ is non-torsion over $vK$ or for some $d\in K$ and $e\in\N$, 
$vd(x-a)^{\rm e}=0$ and $d(x-a)^{\rm e}v$ is transcendental over $Kv$.
\end{lemma}
\begin{proof}
By \cite[part 5) of Lemma 2.8]{K66}, $v(x-a)$ is the maximal element of $v(x-\widetilde{K})$ 
if and only if $v(x-a)\notin v\widetilde{K}$ or $v(x-a)\in v\widetilde{K}$ and 
$\tilde{d}(x-a)v\notin \widetilde{K}v$ for every $\tilde{d}\in\widetilde{K}$ such that 
$v\tilde{d}(x-a)=0$. Note that since $v\widetilde{K}$ is the 
divisible hull of $vK$, $v(x-a)\notin v\widetilde{K}$ if and only if $v(x-a)$ is non-torsion 
over $vK$.

Assume that a) holds. If $v(x-a)$ is non-torsion over $vK$, then b) holds. If 
$v(x-a)\in v\widetilde{K}$, we proceed as follows. We choose $e\in\N$ such that $ev(x-a)\in vK$.
Then there is some $d\in K$ such that $vd=-ev(x-a)=v(x-a)^{\rm e}$. We pick $\tilde{d}
\in\widetilde{K}$ such that $\tilde{d}^e=d$. It follows that $v\tilde{d}(x-a)=0$, and from what 
we said above we know that $\tilde{d}(x-a)v$ is transcendental over $\widetilde{K}v=
\widetilde{Kv}$. Therefore, also $(\tilde{d}(x-a))^e=d(x-a)^e v$ is transcendental over 
$Kv$, and we have shown that b) holds.

\pars
Now assume that b) holds. If $v(x-a)$ is non-torsion over $vK$, then $v(x-a)\notin 
v\widetilde{K}$, and by the equivalence stated above, a) holds. If for some $d\in K$ and 
$e\in\N$, $vd(x-a)^{\rm e}=0$ and $d(x-a)^{\rm e}v$ is transcendental over $Kv$, then 
we proceed as follows. Let $\tilde{d}\in\widetilde{K}$ be such that $v\tilde{d}(x-a)=0$. 
Then $v\tilde{d}^e=-v(x-a)^e=vd$, so that $v\tilde{d}^e d^{-1}=0$ and $0\ne\tilde{d}^e 
d^{-1}v\in \tilde{Kv}$. As $d(x-a)^{\rm e}v$ is transcendental over $Kv$, so are 
$(\tilde{d}^e d^{-1})v\cdot d(x-a)^{\rm e}v=\tilde{d}^e (x-a)^{\rm e}v=
(\tilde{d}(x-a)v)^e$ and $\tilde{d}(x-a)v$. That is, $\tilde{d}(x-a)v\notin \widetilde{K}v$
and again by the above equivalence, a) holds.
\end{proof}

The following is \cite[Lemma 3.7]{KTrans}:
\begin{lemma}                         \label{3.7KTrans}
Assume that the extension $(L(x)|L,v)$ is weakly pure. If we take any
extension of $v$ to $\widetilde{L(x)}$ and take $L^h$ to be the
henselization of $L$ in $(\widetilde{L(x)},v)$, then $L^h$ is the
implicit constant field of this extension:
\[
L^h\>=\>\ic (L(x)|L,v)\>.
\] 
\end{lemma}

\begin{lemma}                         \label{WP2}
Assume that $(L|K,v)$ is a tame extension, $x$ is transcendental over $L$, and we have
an extension $(L(x)|L,v)$ which is weakly pure in $x$. 
\sn
1) \ If $K'|K$ is a subextension of $L|K$ and $v(x-K')$ has no maximal element, then $x$ 
is the limit of a pseudo Cauchy sequence of transcendental type in $(K',v)$. 
\sn
2) \ Assume that $v$ is extended to $\widetilde{K}(x)$. Then every maximal element 
$v(x-a)$ of $v(x-L)$ is also a maximal element of $v(x-\widetilde{K})$.
\sn
3) \ We have that
\[
\ic(K(x)|K,v)\>\subseteq\>\ic(L(x)|L,v)\>=\>L^h\>. 
\]
\sn
4) \ If $(L'|L,v)$ is an algebraic extension, then for every extension of $v$ from $L(x)$
to $L'(x)$, also $(L'(x)|L',v)$ is weakly pure in $x$. 
\end{lemma}
\begin{proof}
1): \ If $K'|K$ is a subextension of $L|K$ and $v(x-K')$ has no maximal element, then by
part 3) of Lemma~\ref{nomax}, $x$ is the limit of a pseudo Cauchy sequence 
$(c_\nu)_{\nu<\lambda}$ in $(K',v)$ without a limit in $K'$. Suppose that it is of 
algebraic type. As it is also a pseudo Cauchy sequence in $(L,v)$ with $x$ as its limit,
it must have a limit $y$ in $L$. Indeed, otherwise $v(x-L)$ would not have a maximum, so by 
our assumption on the extension $(L(x)|L,v)$, $x$ would have to be the limit of 
a pseudo Cauchy sequence of transcendental type in $(L,v)$. However, $x$ cannot be
simultaneously the limit of a pseudo Cauchy sequence of transcendental type in $(L,v)$
and a pseudo Cauchy sequence of algebraic type in $(L,v)$ without a limit in $L$ (as 
follows from Theorems~3 and~4 of \cite{Ka} or from the classification of immediate 
approximation types in \cite{K66}).

Now we have that $v(y-K')$ has no maximum, which is a contradiction to part 1) of 
Corollary~\ref{maxapprcor} since by part 4) of Lemma~\ref{tameprop} the extension 
$(L|K',v)$ is tame.

\mn
2): \ If $v(x-L)$ has a maximal element $v(x-a)$ with $a\in L$, then $x$ is not the 
limit of a pseudo Cauchy sequence in $(L,v)$ without a limit in $L$. Hence by our 
assumption on the extension $(L(x)|L,v)$, the set $v(x-\widetilde{K})$ must have a maximal
element $v(x-a')$ with $a'\in L$. Then $v(x-a')\geq v(x-a)\geq v(x-a')$, so the values
are equal. This proves our assertion.

\mn
3): \ This holds since $\ic(K(x)|K,v)=K(x)^h\cap\widetilde{K}\subseteq L(x)^h\cap
\widetilde{K}=L(x)^h\cap\widetilde{L}$, using also Lemma~\ref{3.7KTrans}.

\mn
4): \ If there is $a\in L$ such that $v(x-a)=\max v(x-\widetilde{L})$, then our assertion is
trivially true since $L\subset L'$ and $\widetilde{L}=\widetilde{L'}$. Now assume that $x$
is the limit of a pseudo Cauchy sequence $(c_\nu)_{\nu<\lambda}$ of transcendental type
in $(L,v)$. Then $(c_\nu)_{\nu<\lambda}$ is also a pseudo Cauchy sequence in $(L',v)$. 
Suppose it were of algebraic type. Then by \cite[Theorem 3]{Ka} there would exist an
algebraic extension $(L'(y),v)$ of $(L',v)$ such that $y$ is a limit of 
$(c_\nu)_{\nu<\lambda}$. However, $y$ is also algebraic over $L$, which leads to a
contradiction, as follows from \cite[Theorem 4]{Ka}.
\end{proof}

\mn
%
%
%
\subsection{Krasner constant and Krasner's Lemma}     \label{sectKras}
\mbox{ }\sn
Take any valued field $(K,v)$ and choose some extension of $v$ from $K$ to its algebraic
closure $\widetilde{K}$. If $a\in \widetilde{K}\setminus K$ is not purely inseparable over 
$K$, then the \bfind{Krasner constant of $a$ over $K$} is defined as:
\[
\Kras (a,K)\>:=\> \max\{v(\tau a-\sigma a)\mid\sigma,\tau\in \Gal K 
\mbox{\ \ and\ \ } \tau a\ne \sigma a\}\>\in\>v\widetilde{K}\>.
\]
Since all extensions of $v$ from $K$ to $\widetilde{K}$ are conjugate, this does not
depend on the choice of the particular extension of $v$. If $a\in K$, then we set 
$\Kras (a,K):=va$. We note:
\begin{lemma}                         \label{krasKras}
1) \ The definition of $\,\Kras (a,K)$ does not depend on the chosen extension of $v$
from $K$ to $\widetilde{K}$.
\sn
2) \ If the extension $(K(a)|K,v)$ is unibranched, then
\[
\Kras (a,K)\>=\> \max\{v(a-\sigma a)\mid \sigma\in \Gal K\,
\mbox{ and }\, a\ne \sigma a\}
\]
and for all $\sigma\in \Gal K$ such that $a\ne \sigma a$,
\[
va\>\leq\> v(a-\sigma a) \>\leq\> \Kras (a,K)\>.
\]
\end{lemma}
\begin{proof}
1): \ Every other extension of $v$ from $K$ to $\widetilde{K}$ is of the form $v\rho$ 
for some $\rho\in\Gal K$, and 
\begin{eqnarray*}
\Kras (a,K)&=& \max\{v(\tau a-\sigma a)\mid\sigma,\tau\in \Gal K 
\mbox{\ \ and\ \ } \tau a\ne \sigma a\}\\
&=& \max\{v(\rho\tau a-\rho\sigma a)\mid\sigma,\tau\in \Gal K 
\mbox{\ \ and\ \ } \rho\tau a\ne \rho\sigma a\}\\
&=& \max\{v\rho(\tau a-\sigma a)\mid \sigma,\tau\in \Gal K\,
\mbox{ and }\, \tau a\ne \sigma a\} \>.
\end{eqnarray*}
2): For a unibranched extension $(K(a)|K,v)$ and every
$\tau\in \Gal K$ we have that $v=v\tau$ on $K(a)$, whence
\begin{eqnarray*}
\Kras (a,K)&=& \max\{v(\tau a-\sigma a)\mid\sigma,\tau\in \Gal K 
\mbox{\ \ and\ \ } \tau a\ne \sigma a\}\\
&=& \max\{v\tau(a-\tau^{-1}\sigma a)\mid\sigma,\tau\in \Gal K 
\mbox{\ \ and\ \ } a\ne \tau^{-1}\sigma a\}\\
&=& \max\{v(a-\sigma a)\mid \sigma\in \Gal K\,
\mbox{ and }\, a\ne \sigma a\} \>,
\end{eqnarray*}
and the inequality $va\leq v(a-\sigma a)$ follows from the fact 
that $va=v\sigma a$.
\end{proof}

\begin{lemma}                               \label{LemKrasKL}
If $a\in \widetilde{K}$ and $(L(a)|K(a),v)$ is any valued field 
extension, then 
\begin{equation}                               \label{KrasKL}
\Kras (a,L)\>\leq\>\Kras (a,K)\>. 
\end{equation}
\end{lemma}
\begin{proof}
Since $a$ is separable-algebraic over $K$, it is also separable-algebraic over $L$.
If $\sigma\in \Gal L$, then $\sigma|_{K\sep}\in \Gal K$; therefore,
\begin{eqnarray*}
\lefteqn{\{v(\tau a-\sigma a)\mid\sigma,\tau\in\Gal L \mbox{ and } 
\tau a\ne\sigma a\}} \qquad\qquad
&&\\
& \subseteq & \{v(\tau a-\sigma a)\mid \sigma,\tau\in \Gal K
\mbox{ and } \tau a\ne \sigma a\}\>.
\end{eqnarray*}
This implies inequality (\ref{KrasKL}).
\end{proof}

We will employ the following variant of Krasner's Lemma:
\begin{proposition}                        \label{vkras}
Take $K(a)|K$ to be a separable-algebraic extension, and
$(K(a,b),v)$ to be any valued field extension of $(K(a),v)$ such that
\begin{equation}
v(b-a)\> >\>\Kras (a,K)\>.
\end{equation}
Then every henselization of $(K(b),v)$
contains $a$.
\end{proposition}
\begin{proof}
Take any extension of $v$ from $K(a,b)$ to $\widetilde{K(b)}$, and take $K(b)^h$
to be the henselization of $(K(b),v)$ with respect to this extension. Then by 
Lemma~\ref{LemKrasKL} applied to $L=K(b)^h$, $\Kras (a,K(b)^h)\leq\Kras 
(a,K)$. Hence by assumption, $v(b-a)>\Kras (a,K(b)^h)$. We
will show that $a$ is fixed by every automorphism $\rho\in \Gal K(b)^h$;
since $a$ is also separable-algebraic over $K(b)^h$,
this will yield that $a\in K(b)^h$. Note that $\rho b=b$ because of 
$\rho\in \Gal K(b)^h$. 

Since $(K(b)^h,v)$ is henselian, using the assumption we may
compute:
\[
v(b-\rho a)\>=\>v\rho (b-a)\> =\>v(b-a)
\> >\>\>\Kras (a,K)\>.
\]
In view of Lemma~\ref{LemKrasKL}, it follows that
\begin{eqnarray*}
v(a -\rho a)&\geq& \min\{v(b-a),v(b-\rho a)\} \> >\> 
\Kras (a,K)\>\geq\> \Kras (a,K(b)^h)\\
&\geq& \max\{v(a-\sigma a)\mid \sigma\in \Gal K\,
\mbox{ and }\, a\ne \sigma a\}\>,
\end{eqnarray*}
which yields that $a=\rho a$.
\end{proof}

\bn
%
%
%
\section{Key sequences}                               \label{sectks}
%

%
%
\subsection{Homogeneous approximations}
\mbox{ }\sn
In this section we will lay out an improved version of the theory of homogeneous elements 
and approximations that was introduced in \cite{KTrans} (cf.\ also \cite{BK}). In contrast 
to those
articles, we will work exclusively with what there was called {\it strongly homogenous 
elements}, and we will strengthen the definition of {\it homogeneous approximations}
accordingly. However, in order to simplify notation, we will drop the word ``strongly''.

\pars
We will say that an element $a$ is \bfind{homogeneous over $(K,v)$} (or just 
{\bf over}  $K$ when it is clear which valuation we refer to) if $a\in K\sep$,
the extension $(K(a)|K,v)$ is unibranched, and
\[
va\>=\>\Kras (a,K)\>.
\]
Note that if $a\in K$, then $\Kras (a,K)=va$ by definition, 
so $a$ is homogeneous over~$K$.

Take a second element $b$ in some algebraically closed valued field extension $(L,v)$ of 
$(K,v)$. We will say that $a$ is a \bfind{homogeneous approximation of $b$ over $K$} 
if $a$ is homogeneous over $K$ and $v(b-a)> vb$. From this it follows that 
$va=vb$ and $v(b-a)>\Kras(a,K)$. 

\begin{remark}
We note that our present use of ``homogeneous approximation'' is almost the same as in the 
definition of homogeneous sequences in \cite{KTrans,BK}, except that now all homogeneous
approximations are taken over $K$; this is a stronger condition.
\end{remark}

\begin{lemma}                              \label{homup}
1) \ The above definition of ``homogeneous element'' does not depend on the extension of 
$v$ from $K$ to $K(a)$.
\sn
2) \ Assume that $a$ is homogeneous over $K$. Then
\begin{equation}                          \label{va=}
va=v(\tau a-\sigma a)\; \mbox{ for all }
\sigma,\tau\in \Gal K \mbox{ such that } \sigma a\ne \tau a\>,  
\end{equation}
and if $(L(a)|K(a),v)$ is any valued field extension, then $a$ is
also homogeneous over $L$.
\sn
3) \ Take elements $a,b,b'$ in some algebraically closed valued field extension $(L,v)$ 
of $(K,v)$. If $a$ is a homogeneous approximation of $b$ over $K$ and if $v(b-b')\geq 
v(b-a)$, then $a$ is also a homogeneous approximation of $b'$ over $K$.
\sn
4) \ If $a$ is a homogeneous approximation of $b$ over $K$ and $\sigma\in
\Gal K$ such that $a\ne\sigma a$, then $v(b-\sigma a)<
v(b-a)$.
\end{lemma}
\begin{proof}
Assertion 1) follows from part 1) of Lemma~\ref{krasKras}
and the fact that every other extension of $v$ from $K$ to 
$K(a)$ is of the form $v\circ\sigma$ for some $\sigma\in \Gal K$. In order to prove 
assertion 2), assume that $a$ is homogeneous over $K$. Then $v\tau a=va=v\sigma a$ since
the extension $(K(a)|K,v)$ is unibranched, whence $va\leq v(\tau a-\sigma a)
\leq\Kras (a,K)=va$ for all $\sigma,\tau\in \Gal K$ such that $\sigma 
a\ne \tau a$. Hence equality holds everywhere, which proves (\ref{va=}).
If $(L(a)|K(a),v)$ is any valued field extension, then $a\in
L\sep$, $v\sigma a=v\sigma|_{K\sep}a=va$ for all $\sigma\in \Gal L$, 
and for every $\sigma,\tau\in\Gal L$ with $\tau a\ne\sigma a$, we have that
\[
v(\tau a-\sigma a)\>=\> v(\tau|_{K\sep}a-\sigma|_{K\sep} a)
\>=\> va\>,
\]
whence $\Kras (a,L)=va$.
\mn
3): \ By assumption we have that $v(b'-a)\geq\min
\{v(b-b'),v(b-a)\} =v(b-a)>vb=vb'$. This yields the assertion, since 
$a$ is homogeneous over $K$.
\mn
4): \ Since $a$ is homogeneous over $K$, we have that 
$v(a-\sigma a)=va<v(b-a)$ 
and therefore, $v(b-\sigma a)=\min\{v(b-a),v(a
-\sigma a)\}=v(a-\sigma a)<v(b-a)$.
\end{proof}

The following important property of 
homogeneous approximations is a consequence of Proposition~\ref{vkras}:
\begin{lemma}                               \label{krlkra}
If $a$ is a homogeneous approximation of $b$ over $K$, then $a$ lies in the
henselization of $K(b)$ w.r.t.\ every extension of the valuation $v$
from $K(a,b)$ to $\widetilde{K(b)}$.
\end{lemma}

The following gives a crucial criterion for an element to be homogeneous over $K$:
\begin{lemma}                              \label{lstronghom}
Suppose that $a\in\widetilde{K}$ and that there is some extension of $v$
from $K$ to $K(a)$ such that if {\rm e} is the least positive integer
for which ${\rm e}va\in vK$, then
\sn
a) \ {\rm e} is not divisible by $\chara Kv$,\n
b) \ there exists some $c\in K$ such that $vca^{\rm e}=0$, $ca^{\rm e}v$
is separable-algebraic over $Kv$, and the degree of $ca^{\rm e}$ over
$K$ is equal to the degree {\rm f} of $ca^{\rm e}v$ over $Kv$.
\sn
Then $(K(a)|K,v)$ is a separable unibranched extension of degree $\rme \cdot \rmf$, 
and $a$ is homogeneous over $K$.
\end{lemma}
\begin{proof}
We have
\begin{eqnarray*}
\rme \cdot \rmf & \geq & [K(a):K(a^{\rm e})]\cdot [K(a^{\rm e}):K]
\>=\> [K(a):K] \\
 & \geq & (vK(a):vK)\cdot [K(a)v:Kv]\>\geq\>\rme \cdot \rmf\>.
\end{eqnarray*}
Hence equality holds everywhere, and we find that $[K(a):K]=\rme \cdot \rmf$, 
$(vK(a):vK)={\rm e}$ and $[K(a)v:Kv]={\rm f}$. By the fundamental 
inequality, this implies that the extension $(K(a)|K,v)$ is unibranched. By 
assumption a), the extension $K(a)|K(a^{\rm e})$ is separable. 
By assumption b), the residue field extension $K(a^{\rm e})v|Kv$ is separable of 
degree $[K(a^{\rm e}):K]$, which shows that $K(a^{\rm e})|K$ must be separable.
Altogether, we find that $K(a)|K$ is separable.

In order to show that $a$ is homogeneous over $K$,
we may assume that $a\notin K$. Take $\sigma\in \Gal K$ with 
$\sigma a\ne a$ and set $\eta:= \sigma a/a\ne 1$. 
If $\sigma a^{\rm e} \ne a^{\rm e}$, then $\sigma ca^{\rm e}
=c\sigma a^{\rm e}\ne ca^{\rm e}$ and by hypothesis, their residues 
are also distinct, so the residue of $\sigma a^{\rm e}/a^{\rm e}=
\eta^{\rm e}$ is not $1$. It then follows that the residue of $\eta$ is not $1$. If
$\sigma a^{\rm e} = a^{\rm e}$, then $\eta$ is an e-th root of unity. Since e
is not divisible by the residue characteristic, it again follows that
the residue of $\eta$ is not equal to $1$. Hence in both cases, we obtain that 
$v(\eta-1)=0$, which shows that $v(\sigma a - a)= va$. We have 
now proved our lemma.        
\end{proof}

\begin{lemma}                               \label{exkrap}
Assume that $b$ is an element in some algebraically closed valued field
extension $(L,v)$ of $(K,v)$. Assume further that there are ${\rm e}\in\N$
not divisible by $\chara Kv$ and $c\in K$ such that $vcb^{\rm e}
=0$ and $cb^{\rm e}v$ is separable-algebraic of degree \rmf over $Kv$. Then we can find
a homogeneous approximation $a\in \tilde K$ of $b$ over $K$, of degree $\rme
\cdot\rmf$ over $K$.

If $\tilde{a}\in\tilde K$ satisfies $v(b-\tilde{a})>vb$, then $[K(\tilde{a}):K]\geq \rme
\cdot\rmf$; if in addition $\tilde{a}\in K(a)$, then $\tilde{a}$ is also a homogeneous
approximation of $b$ over~$K$, $K(\tilde{a})=K(a)$, and $\Kras(\tilde{a},K)=\Kras(a,K)$.
\end{lemma}
\begin{proof}
Take a monic polynomial $g$ over $K$ with $v$-integral coefficients
whose reduction modulo $v$ is the minimal polynomial of $cb^{\rm e}v$
over $Kv$. Then let $a_0\in\widetilde{K}$ be the root of $g$ whose residue
is $cb^{\rm e}v$. The degree of $a_0$ over $K$ is the same as that of
$cb^{\rm e}v$ over $Kv$. We have that $v(\frac{a_0}{cb^{\rm e}}-1)>0$.
So there exists $a_1\in \widetilde{K}$ with residue $1$ and such that
$a_1^{\rm e}= \frac{a_0}{cb^{\rm e}}$. Then for $a:=
a_1 b$, we find that $v(b-a)=vb+ v(a_1-1)>vb$ and 
$ca^{\rm e}=a_0\,$, hence $a\in\tilde K$. It follows that 
$va=vb$ and $ca^{\rm e}v=cb^{\rm e}v$. By the foregoing lemma, this 
shows that $a$ is homogeneous over $K$.
\pars
Now assume that also $\tilde{a}\in\tilde K$ satisfies $v(b-\tilde{a})>vb$. Then 
$v\tilde{a}=vb$, whence $(vK(\tilde{a}):vK)\geq\rme$. Further, $v(cb^{\rm e}-
\tilde{a}^{\rm e})>0$, whence $c\tilde{a}^{\rm e}v=cb^{\rm e}v$
and $[K(\tilde{a})v:Kv]\geq\rmf$. Therefore, $[K(\tilde{a}):K]\geq\rme\rmf$. We note that
$v(\tilde{a}-a)\geq\min\{v(b-a),v(b-\tilde{a})\}>vb=va=v\tilde{a}$.
\pars
Finally, assume in addition that $\tilde{a}\in K(a)$. Then 
\[
\rme\cdot\rmf\>\leq\> [K(\tilde{a}):K] \>\leq\> [K(a):K] \>=\> \rme\cdot\rmf\>,
\]
showing that $K(\tilde{a})=K(a)$. Thus for every $\sigma\in\Gal K$ we have that 
$\sigma \tilde{a}\ne \tilde{a}$ if and only if $\sigma a\ne a$. If this is the case, then 
\[
v(\tilde{a}-\sigma\tilde{a})\>=\> v(\tilde{a}-a+a-\sigma a+\sigma a-\sigma\tilde{a})\>.
\]
As $(K(a)|K,v)$ is unibranched and $K(\tilde{a})= K(a)$, we know that 
\[
v(\sigma a-\sigma \tilde{a})\>=\>v\sigma(a-\tilde{a})\>=\>v(a-\tilde{a})\> >\>va
\>=\>\Kras(a,K)\>\geq\>v(a-\sigma a)\>.
\]
Consequently, $v(\tilde{a}-\sigma \tilde{a})=v(a-\sigma a)$ and therefore, 
$\Kras(\tilde{a},K)=\Kras(a,K)=va=v\tilde{a}$. This proves that also $a$ is homogeneous 
over $K$.
\end{proof}

\mn
%
%
%
\subsection{Key sequences}
\mbox{ }\sn
We will work with a variant of the notion of ``homogeneous sequence'' which was introduced 
in \cite{KTrans}, and the stronger notion of ``key sequence''. Let $(K(x)|K,v)$ be any
extension of valued fields. We fix an extension of $v$ to $\widetilde{K(x)}$.

\pars
Let $S$ be an initial segment of $\N$, that is, $S=\N$ or
$S=\{1,\ldots,n\}$ for some $n\in\N$. A sequence
\[
{\eu S}\>:=\> (a_i)_{i\in S}
\]
of elements in $\widetilde{K}$  will be called a \bfind{key sequence for $(K(x)|K,v)$} if
\sn
{\bf (KS1)} \ \ $a_1\in K$, 
\sn
{\bf (KS2)} \ \ if $1<i\in S$, then $K(a_i)=K(\tilde{a}_i)$ and $v(x-a_i)\geq 
v(x-\tilde{a}_i)$ for some $\tilde{a}_i$
such that $\tilde{a}_i-a_{i-1}$ is a homogeneous approximation of $x-a_{i-1}$ over $K$,
\sn
{\bf (KS3)} \ \ $v(x-a_i)=\max v(x-K(a_i))$, unless $v(x-K(a_i))$ has no largest element,
in which case $i$ is the last element of $S$.

\pars
Further, ${\eu S}$ will be called a \bfind{pre-complete key sequence} if in addition to 
the above, $S=\N$ or the following conditions are satisfied:
\sn
{\bf (CKS1)} \ \ if $x\in\widetilde{K}$, then $S$ is finite, and if $n$ is its last element, 
then $a_n=x$,
\sn
{\bf (CKS2)} \ \ if $x$ is transcendental over $K$ and $S= \{1,\dotsc,n\}$, then 
either $v(x-a_n)=\max v(x-\widetilde{K})$, or $x$ is limit of a pseudo Cauchy 
sequence of transcendental type in $(K(a_n),v)$,
\sn
and it will be called \bfind{complete} if the second case in (CKS2) does not appear.

\sn
We call $S$ the \bfind{support} of the sequence $\eu S$. We set
\[
K_{\eu S}\>:=\>K(a_i\mid i\in S)\>.
\]
If ${\eu S}$ is the empty sequence, then $K_{\eu S}=K$.

From the above definitions, the following is obvious:
\begin{lemma}                               \label{S'S}
Assume that $S'\subseteq S$ are initial segments of $\N$. If $(a_i)_{i\in S}$ 
is a key sequence for $(K(x)|K,v)$, then so is $(a_i)_{i\in S'}$.
\end{lemma}

\pars
From the definition of the key sequence $(a_i)_{i\in S}$, we obtain:
\begin{lemma}                               \label{kspcs}
Take a key sequence $(a_i)_{i\in S}$. Then the following statements hold:
\n
1) \ For every $i\in S$, $a_i\in K\sep$.
\sn
2) \ For every $i\in S$, 
\begin{equation}                           \label{>}
v(x-a_i)\> >\> v(x-a_{i-1})\quad\mbox{ if } 1<i\in S
\end{equation}
and 
\begin{equation}                           \label{notin}
a_i\>\notin\> K(a_{i-1})\>.
\end{equation}
\sn
3) \ If $i,j\in S$ with $i<j$, then
\begin{equation}                            \label{pcs}
v(x-a_j)\> >\>v(x-a_i)\>=\>v(a_{i+1}-a_i)\>.
\end{equation}
If $S=\N$, then $(a_i)_{i\in S}$ is a pseudo Cauchy sequence in
$K_{\eu S}$ with pseudo limit $x$.
\end{lemma}
\begin{proof}
1): \ We proceed by induction on $i\in S$. By (KS1), $a_1\in K\subseteq K\sep$.
Assume that $1<i\in S$ and we have already shown that $a_{i-1}\in K\sep$. 
By (KS2), $\tilde{a}_i-a_{i-1}\in K\sep$, whence $\tilde{a}_i\in K\sep$ and
$a_i\in K(\tilde{a}_i)\subseteq K\sep$. 

\sn
2): \ By (KS2), for $1<i\in S$ we have:
\[
v(x-a_i)\>\geq\>v(x-\tilde{a}_i)\>=\>v(x-a_{i-1}-(\tilde{a}_i-a_{i-1}))
\> >\>v(x-a_{i-1})\>.
\]
This proves (\ref{>}). Assertion (\ref{notin}) follows from (KS3) since $v(x-a_i)>
v(x-a_{i-1})=\max v(x-K(a_{i-1}))$.

\sn
3): \ Take $i,j\in S$ with $1\leq i<j$. Then (\ref{pcs}) follows by induction from (\ref{>}). 
Now assume that $S=\N$. Then for all $k>j>i\geq 1$,
\[
v(x-a_k)\> >\>v(x-a_j) \> >\>v(x-a_i)
\]
and therefore,
\begin{eqnarray*}
v(a_k-a_j) & = & \min\{v(x-a_k),v(x-a_j)\}\>=\>v(x-a_j)\\
 & > & v(x-a_i) \>=\> \min\{v(x-a_j),v(x-a_i)\}\>=\>v(a_j-a_i)\>.
\end{eqnarray*}
This shows that $(a_i)_{i\in S}$ is a pseudo Cauchy sequence. The
equality in (\ref{pcs}) shows that $x$ is a pseudo limit of this sequence.
\end{proof}

Let us also observe the following:
\begin{lemma}                               \label{xx}
Let $x,z$ be elements in some valued field extension of $(K,v)$ that contains $\tilde K$.
Take a key sequence $(a_i)_{i\in S}$ for $(K(x)|K,v)$. Then the following assertions hold:
\sn
1) \ Assume that $v(x-z)> v(x-a_i)$ for all $i\in S$. Then $(a_i)_{i\in S}$ is also a key
sequence for $(K(z)|K,v)$.
\sn
2) \ Take $k\in S$. Then $(a_i)_{i\leq k}$ is a complete key sequence for $(K(a_k)|K,v)$. 
\end{lemma}
\begin{proof}
1): \ It follows from part 3) of Lemma~\ref{homup} that if (KS2) holds, then it also holds
with $z$ in place of $x$. We now show the same for (KS3). Since $v(x-z)> 
v(x-a_i)$, it follows that $v(z-a_i)=\min\{v(x-a_i),v(x-z)\}=v(x-a_i)$. In order to 
complete our proof, we need to show that $v(z-y)\leq v(z-a_i)$
for $y\in K(a_i)$. Suppose otherwise. Then $v(x-y)\geq\min\{v(x-z),v(z-y)\}>
v(x-a_i)=\max v(x-K(a_i))$, contradiction. This shows that (KS3) also holds 
with $z$ in place of $x$.

\sn
2): \ By Lemma~\ref{S'S}, $(a_i)_{i< k}$ is a key sequence for $(K(x)|K,v)$. Since 
$v(x-a_k)>v(x-a_i)$ for $i<k$, part 1) shows that it is also a key sequence for 
$(K(a_k)|K,v)$. Next, we show that (KS2) also holds for $a_k$ in place of $x$
and $k$ in place of $i$. 
Since $(a_i)_{i\in S}$ is a key sequence for $(K(x)|K,v)$, we know already that 
$K(a_k)=K(\tilde{a}_k)$, and it remains to show that $a:=\tilde{a}_k-a_{k-1}$ is a 
homogeneous approximation of $b':=a_k-a_{k-1}$. With $b:=x-a_{k-1}$, we compute:
\[
v(b-b')\>=\> v(x-a_k) \>\geq\> v(x-\tilde{a}_k)\>=\> v(b-a)\>.
\]
Hence the required result follows from part 3) of Lemma~\ref{homup}. Also (KS3) is 
satisfied for $a_k$ in place of $x$ and $k$ in place of $i$ because $v(a_k-a_k)=\infty=
\max v(x-K(a_k))$. This also proves that by definition, $(a_i)_{i\leq k}$
is a complete key sequence for $(K(a_k)|K,v)$.
\end{proof}

What is special about key sequences is described by the following lemma:
\begin{lemma}                               \label{KSinL}
Assume that $(a_i)_{i\in S}$ is a key sequence for $(K(x)|K,v)$. Then
\begin{equation}                            \label{KSin}
K_{\eu S} \>\subseteq\> K(x)^h \>.
\end{equation}
For every $n\in S$, $a_1,\ldots,a_n\in K(a_n)^h$. 
If $S=\{1,\ldots,n\}$, then
\begin{equation}                            \label{in}
K_{\eu S}^h \>=\> K(a_n)^h\>.
\end{equation}
Hence if $(K,v)$ is henselian, then $K_{\eu S}=K(a_n)$.
\end{lemma}
\begin{proof}
By induction on $i\in S$, we show that $a_i\in K(x)^h$. By (KS1), $a_1\in
K\subseteq K(x)^h$. Assume that $1<i\in S$ and we have already shown that $a_{i-1}
\>\in\>K(x)^h$. As $\tilde{a}_i-a_{i-1}$ is a homogeneous approximation of $x-a_{i-1}\in
K(x)^h$ over $K$ for $\tilde{a}_i$ as in (KS2), we know from Lemma~\ref{krlkra} that
\[
\tilde{a}_i-a_{i-1}\>\in\>K(x-a_{i-1})^h \>\subseteq\> K(x)^h
\]
and hence also $\tilde{a}_i\in K(x)^h$. This implies that $a_i\in K(\tilde{a}_i)\subset
K(x)^h$, which proves (\ref{KSin}). 

Now all other assertions follow when we replace $x$ by $a_n$ in the above argument,
using the fact that by Lemma~\ref{S'S} with $S'=\{1,\ldots,n\}$, 
$(a_i)_{i\leq n}$ is a key sequence for $(K(a_n)|K,v)$.
\end{proof}

\begin{lemma}                           \label{lemKS4}
Every key sequence $(a_i)_{i\in S}$ has property 
\sn
{\bf (KS4)} \ \ $K(a_{i-1})^h\subsetneq K(a_i)^h$ if $1<i\in S$.
\end{lemma}
\begin{proof}
Suppose that $1<i\in S$ and $K(a_{i-1})^h=K(a_i)^h$. Then $a_i\in K(a_{i-1})^h$ and since
$(K(a_{i-1})^h|K(a_{i-1}),v)$ is immediate, by part 1) of Lemma~\ref{nomax}
the set $v(a_i-K(a_{i-1}))$ has no maximum. 
Hence there exists $c\in K(a_{i-1})$ such that $v(a_i-c)>v(a_i-a_{i-1})$. On the other 
hand, $v(x-a_i)>v(x-a_{i-1})$ by Lemma~\ref{kspcs}, whence 
\[
v(a_i-a_{i-1})\>=\>\min\{v(x-a_i),v(x-a_{i-1})\}\>=\>v(x-a_{i-1})\>. 
\]
We conclude that
\[
v(x-c)\>=\>\min\{v(x-a_i),v(a_i-c)\}\> >\>v(x-a_{i-1}) 
\]
in contradiction to (KS3). This proves our assertion.
\end{proof}

\begin{proposition}                     \label{KS5,6}
Take a henselian field $(K,v)$.
Assume that $(a_i)_{i\in S}$  is a key sequence for $(K(x)|K,v)$. Then it has the 
following additional properties:
\sn
{\bf (KS5)} \ \ $\Kras(a_i,K)=\Kras(a_i-a_{i-1},K)=v(x-a_{i-1})$ if $1<i\in S$,
\sn
{\bf (KS6)} \ \ if $1<i\in S$ and $z\in\tilde K$ such that $v(x-z)> v(x-a_{i-1})$, then
$[K(z):K]\geq [K(a_i):K]$.
\end{proposition}

This proposition follows by induction on $i\in S\setminus\{1\}$ from a slightly more 
general result:
\begin{lemma}                     \label{KS4,5lem}
Take a henselian field $(K,v)$. Assume that $i>1$, $(a_j)_{j<i}$  is a key sequence for 
$(K(x)|K,v)$ with properties (KS5) and (KS6)), that $\tilde{a}_i-a_{i-1}$ is a 
homogeneous approximation of $x-a_{i-1}\,$,
and that $a_i\in K(\tilde{a}_i)$ with $v(x-a_i)\geq v(x-\tilde{a}_i)$. Then $K(a_i)=
K(\tilde{a}_i)$, and the sequence $(a_j)_{j\leq i}$ has properties (KS5) and (KS6).
\end{lemma}
\begin{proof}
We will first prove (KS5) in the special case where $a_i=\tilde{a}_i\,$. 
We start with the case of $i=2$, where $a_{i-1}=a_1\in K$ by (KS1). Hence for every 
$\sigma\in \Gal K$, 
\[
v(\tilde{a}_2-\sigma a_2)\>=\> v(\tilde{a}_2-a_1-\sigma (\tilde{a}_2-a_1))\>.
\]
This implies that $\Kras(\tilde{a}_2,K)=\Kras(\tilde{a}_2-a_1,K)=v(\tilde{a}_2-a_1)$. 

Now we consider the case of $2<i\in S$. By the assumption of our lemma, (KS5) holds with 
$i-1$ in place of $i$. We wish to compute 
$v(\tilde{a}_i-\sigma \tilde{a}_i)$ whenever $\sigma\in \Gal K$ with $\tilde{a}_i\ne\sigma
\tilde{a}_i\,$. We set $a:=\tilde{a}_i-a_{i-1}\,$. Since $\tilde{a}_i=a_{i-1}+a$, we must 
have that $a_{i-1}\ne\sigma a_{i-1}$ or $a\ne\sigma a$, and
\[
v(\tilde{a}_i-\sigma \tilde{a}_i)\>=\> v(a_{i-1}-\sigma a_{i-1} + a-\sigma a)\>.
\]
If $a_{i-1}\ne\sigma a_{i-1}$, then 
\[
v(a_{i-1}-\sigma a_{i-1})\>\leq\> \Kras(a_{i-1},K)\>=\> v(x-a_{i-2})\><\>
v(x-a_{i-1})\>,
\]
where we have used (\ref{>}) for the last inequality. 
If $a\ne\sigma a$, then $v(a-\sigma a)=va=v(x-a_{i-1})$ since $a$ is a homogeneous
approximation of $x-a_{i-1}$ over $K$. In both cases,
\begin{equation}                         \label{bothc}
v(\tilde{a}_i-\sigma \tilde{a}_i)\>=\> \min\{v(a_{i-1}-\sigma a_{i-1}),v(a-\sigma a)\}
\>\leq\> v(x-a_{i-1})\>. 
\end{equation}
On the other hand, as $v(x-\tilde{a}_i)>v(x-a_{i-1})=\max v(x-K(a_{i-1}))$ by our 
assumption on $\tilde{a}_i-a_{i-1}$ and (KS3) with $i-1$ in place of $i$, we have 
that $\tilde{a}_i\notin K(a_{i-1})$. Consequently, $K(a_{i-1},\tilde{a}_i)=
K(a_{i-1},a)$ is a nontrivial extension of $K(a_{i-1})$. It is a separable-algebraic 
extension of $K$ since by part 1) of Lemma~\ref{kspcs}, $a_{i-1}$ is separable over $K$,
and the same holds for $a$ as it is homogeneous over $K$. Hence there is $\sigma\in
\Gal K$ such that $a_{i-1}=\sigma a_{i-1}$ and $a\ne\sigma a$, in which case $v(a_i-\sigma
a_i)=v(x-a_{i-1})$. Consequently, $\Kras(\tilde{a}_i,K)=v(x-a_{i-1})$.

\pars
Now we take $a_i\in K(\tilde{a}_i)$ with $v(x-a_i)\geq v(x-\tilde{a}_i)$.
We set $d:=a_i-\tilde{a}_i$ and observe that $vd\geq v(x-\tilde{a}_i)>v(x-a_{i-1})$. As
$(K,v)$ is henselian, we have that 
$vd=v\sigma d$ and hence $v(d-\sigma d)>v(x-a_{i-1})$ for all $\sigma\in \Gal K$.
Assuming that $\sigma\tilde{a}_i\ne\tilde{a}_i\,$, we compute, using (\ref{bothc}):
\[
v(a_i-\sigma a_i)\>=\>v(\tilde{a}_i-\sigma\tilde{a}_i+d-\sigma d)\>=\>
\min\{v(\tilde{a}_i-\sigma\tilde{a}_i),v(d-\sigma d)\}\>=\>v(\tilde{a}_i
-\sigma\tilde{a}_i)\>.
\]
In particular, $\sigma\tilde{a}_i\ne\tilde{a}_i$ implies that $\sigma a_i\ne a_i$. 
Since $a_i\in K(\tilde{a}_i)$ and $K(\tilde{a}_i)|K(a_i)$ is a separable extension, 
we find that $K(a_i)=K(\tilde{a}_i)$. 
Every $\sigma$ such that $\sigma a_i\ne a_i$ also satisfies
$\sigma\tilde{a}_i\ne\tilde{a}_i\,$, so we obtain that
\[
\kras(a_i,K)\>=\>\kras(\tilde{a}_i,K)\>=\>v(x-a_{i-1})\>.
\]

\pars
In order to show that (KS6) holds, assume that $z\in\tilde K$ satisfies 
$v(x-z)> v(x-a_{i-1})$. Since also $v(x-a_i)> v(x-a_{i-1})$, we have that
\[
v(z-a_i)\>\geq\>\min\{v(x-z),v(x-a_i)\}\> >\> v(x-a_{i-1})\>=\> \Kras(a_i,K)\>, 
\]
where the last equation holds by (KS5). Hence from Proposition~\ref{vkras} it follows that
$a_i\in K(z)$, which proves that (KS6) holds. 
\end{proof}

\begin{proposition}                         \label{SNpure}
Take any valued field $(K,v)$ and any extension $(K(x)|K,v)$. 
Assume that ${\eu S}=(a_i)_{i\in \N}$  is a key sequence for $(K(x)|K,v)$. Then ${\eu S}$ 
is a complete key sequence for $(K(x)|K,v)$, $x$ is transcendental over $K$, 
$(a_i)_{i\in\N}$ is a pseudo Cauchy sequence of transcendental type in 
$(K_{\eu S},v)$ with pseudo limit $x$, and $(K_{\eu S}(x)|K_{\eu S},v)$ is immediate.
\end{proposition}
\begin{proof}
${\eu S}$ is complete as it satisfies (CKS2) because its support is $\N$.

By part 2) of Lemma~\ref{kspcs}, $(a_i)_{i\in\N}$ is a pseudo Cauchy sequence 
in $K_{\eu S}$ with pseudo
limit $x$. Suppose $(a_i)_{i\in\N}$ were of algebraic type. Then by \cite[Theorem~3]{Ka}, 
there would exist some algebraic extension $(K_{\eu S}(y)|K_{\eu S},v)$ with $y$ a pseudo 
limit of the sequence. But then $v(x-y) >v(x-a_i)$ for all $i\in\N$ and by part 1) 
of Lemma~\ref{xx}, $(a_i)_{i\in S}$ is also a key sequence for $(K(y)|K,v)$. Hence by
Lemma~\ref{KSinL}, $K_{\eu S} \subset K(y)^h=K^h(y)$. Since $y$ is algebraic over $K$,
the extension $K^h(y)|K^h$ is finite. On the other hand, $K_{\eu S}^h|K^h$ is infinite 
since (KS4) holds. This contradiction shows that the sequence $(a_i)_{i\in\N}$
is of transcendental type. Hence $x$ is transcendental over $K$ and it
follows from Lemma~\ref{limtpcs} that $(K_{\eu S}(x)| K_{\eu S},v)$ is immediate.
\end{proof}

The following is Theorem~5.9 of \cite{KTrans}, reformulated using the notion of 
``pre-complete key sequence''.
\begin{theorem}                             \label{thIC}
Assume that $(K(x)|K,v)$ is a transcendental extension and $\eu S$ is a pre-complete key 
sequence for $(K(x)|K,v)$. Then
\[
K_{\eu S}^h \>=\> \ic (K(x)|K,v)\>.
\]
Further, $K_{\eu S}v$ is the relative algebraic closure of $Kv$ in
$K(x)v$, and the torsion subgroup of $vK(x)/vK_{\eu S}$ is finite.
\end{theorem}

\mn
%
%
%
\subsection{Key sequences and key polynomials}     \label{sectseqpol}
\mbox{ }\sn
\begin{proposition}                        \label{kskp}
Take any valued field $(K,v)$, an extension $(K(x)|K,v)$, and a key sequence ${\eu S}
=(a_i)_{i\in S}$ for $(K(x)|K,v)$ with additional properties (KS5) and (KS6). 
Let $Q_i\in K[X]$ be the minimal polynomial of $a_i$ over $K$. Then $Q_i$ is a
key polynomial and $a_i$ is the unique maximal root of $Q_i\,$. 

Assume in addition that ${\eu S}$ is complete.
Then $(Q_i)_{i\in S}$ is a complete sequence of key polynomials for $(K(x)|K,v)$. 
\end{proposition}
\begin{proof}
By (KS5) and (\ref{>}) we know that whenever $a'_i\ne a_i$ is a conjugate of 
$a_i$ over $K$, then $v(a_i-a'_i)\leq v(x-a_{i-1})<v(x-a_i)$;
this implies that $v(x-a'_i)=\min\{v(x-a_i),v(a_i-a'_i)\}<v(x-a_i)$.
Hence $a_i$ is the unique maximal root of $Q_i\,$. In particular, $\delta(Q_i)=
v(x-a_i)$.

Assume that $f\in K[X]$ such that $\delta(f)\geq \delta(Q_i)=v(x-{a}_i)$. Take a root $z$
of $f$ such that $v(x-z)=\delta(f)\geq v(x-a_i)>v(x-a_{i-1})$. Then by (KS6), 
$\deg f =[K(z):K]\geq [K(a_i):K]=\deg Q_i\,$. This proves that $Q_i$ is a key 
polynomial.


\pars
We turn to the last assertion of our proposition. We note that if $S=\{1,\ldots,n\}$ 
and $x$ is transcendental over $K$, then by our assumption, $v(x-a_n)=\max 
v(x-\widetilde{K})$; this also holds if $x$ is algebraic over $K$, since then by the 
definition of completeness, $x=a_n$ and thus $v(x-a_n)=\infty$.

We wish to show that for every $f(X) \in K[X]$ 
there exists $i\in S$  such that $\deg Q_i \leq \deg f$ and $\delta(f) \leq \delta(Q_i)$. 
Let $z$ be a root of $f$ such that $v(x-z)=\delta(f)$. Take $i\in S$ maximal with 
$\deg Q_i \leq \deg f$. Suppose that $v(x-z)>v(x-a_i)$. Then by what we have said 
above, $i$ cannot be the last element of $S$. By (KS6), $\deg f\geq [K(z):K]\geq 
[K(a_{i+1}):K]=\deg Q_{i+1}$, which contradicts our choice of $i$. This shows that 
$\delta(f)=v(x-z)\leq v(x-a_i) =\delta(Q_i)$.
\end{proof}

\mn
%
%
%
\subsection{Key sequences and tameness}     \label{secteqtame}
%
\mbox{ }\sn
We wish to characterize the extensions $(K(x)|K,v)$ for which there exist
key sequences.

If an element $a\in \widetilde{K}$ satisfies the conditions of
Lemma~\ref{lstronghom}, then $(K(a)|K,v)$ is a tame extension. 
More generally, the following holds.
\begin{proposition}                         \label{hit}
Take any valued field $(K,v)$. If $a$ is homogeneous over $K$,
then $(K(a)|K,v)$ is a tame extension. If $(K,v)$ is henselian and $\eu S$ is a key
sequence for $(K(x)|K,v)$, then $K_{\eu S}$ is a tame extension of $K$.
\end{proposition}
\begin{proof}
If $(K(a)|K,v)$ is not a tame extension, then by part 2) of Lemma~\ref{tameprop}, also 
$(K(a)^h|K^h,v)$ is not a tame extension. Then $a$ does not lie in the ramification field 
$K^r$ of the extension $(K\sep|K^h,v)$ since by \cite[part (b) of Lemma 2.13]{[K7]}, 
$K^r$ is the unique maximal tame extension of $(K^h,v)$. So there exists an automorphism 
$\sigma$ in the ramification group such that $\sigma a\ne a$. But by the
definition of the ramification group,
\[
\Kras (a,K) \>\geq\> v(\sigma a- a)\> >\> va\>,
\]
showing that $a$ is not homogeneous over $K$.

\pars
Now assume that $(K,v)$ is henselian and $\eu S=(a_i)_{i\in S}$ is a key sequence for 
$(K(x)|K,v)$. By induction on $i\in S$, we prove that $(K(a_i)|K,v)$ is a tame extension. 
This holds for $i=1$ since $a_1\in K$. Now assume that we have already shown that 
$(K(a_{i-1}),v)=(K(\tilde{a}_{i-1}),v)$ is a tame extension of $(K,v)$. Since 
$\tilde{a}_i-a_{i-1}$ is homogeneous over $K$, the first part of our proposition shows 
that $(K(\tilde{a}_i-a_{i-1})|K,v)$ is a tame extension. Hence by parts 3) and 4) of
Lemma~\ref{tameprop}, $(K(a_{i-1},\tilde{a}_i-a_{i-1}),v)$ and its subfield $(K(a_i),v)
=(K(\tilde{a}_i),v)$ are tame extensions of 
$(K,v)$. As $K_{\eu S}=K(a_i\mid i\in S)$, we can again employ part 3) of
Lemma~\ref{tameprop} to conclude that $(K_{\eu S}|K,v)$ is a tame extension. 
\end{proof}

We now prove the existence of key sequences under suitable tameness assumptions:
\begin{proposition}                                   \label{existcks}
Take a henselian field $(K,v)$ and an extension $(K(x)|K,v)$. 
If $x\in\widetilde{K}$, then assume that the extension $(K(x)|K,v)$ is tame. If $x$ is
transcendental over $K$, then assume that there is a tame extension $(L'|K,v)$ such that
for a suitable extension of $v$ from $K(x)$ to $L'(x)$, the extension $(L'(x)|L',v)$ is
weakly pure in $x$. Then there exists a pre-complete key sequence ${\eu S}=
(a_i)_{i\in S}$ for $(K(x)|K,v)$. It is complete if $x\in\widetilde{K}$ or 
$(L'(x)|L',v)$ is valuation transcendental.
\end{proposition}
\begin{proof}
If $x\in\widetilde{K}$, then
we set $L'=K(x)$; hence in all cases we have that $(L'|K,v)$ is a tame extension.

If $v(x-K)$ has no maximum, then we set $S=\{1\}$ and $a_1=0$; otherwise, 
we choose $a_1\in K$ such that $v(x-a_1)=\max v(x-K)$. Then (KS1) is 
satisfied and (KS3) holds for $i=1$; (KS2) is void 
for $i=1$. Note that if $x\in\widetilde{K}$, then a maximum always
exists by part 1) of Corollary~\ref{maxapprcor}.

\pars
Now we assume that $i>1$ and a key sequence ${\eu S}_{i-1}
=(a_j)_{1\leq j\leq i-1}$ for the extension $(K(x)|K,v)$ has already been
constructed such that $K_{{\eu S}_{i-1}}\subseteq L'$. By Proposition~\ref{KS5,6}, it
has properties (KS5) and (KS6).

If $x\in K(a_{i-1})$, then from (KS3) it 
follows that $x=a_{i-1}$, i.e., (CKS1) holds for $n=i-1$. In this case, or if 
(CKS2) holds for $n=i-1$, our construction stops here. Otherwise, we proceed as follows.

First we show that $v(x-K(a_{i-1}))$ has a largest element, which is not the 
largest element of $v(x-L')$. Assume that $x$ is transcendental over $K$. If 
$v(x-K(a_{i-1}))$ would not have a largest element, then $x$ would be the limit 
of a pseudo Cauchy sequence in $(K(a_{i-1}),v)$ of transcendental type by part
1) of Lemma~\ref{WP2} and consequently, (CKS2) would hold for $n=i-1$. Hence 
$v(x-K(a_{i-1}))$ has a largest element $v(x-a_{i-1})$. If this would also be the
largest element of $v(x-L')$, then by part 2) of Lemma~\ref{WP2} it would also be the
largest element of $v(x-\widetilde{K})$ and again, (CKS2) would hold for $n=i-1$. 
Now assume that $x$ is algebraic over $K$. As $(L'|K(a_{i-1}),v)$ is a tame extension by 
part 4) of Lemma~\ref{tameprop}, we know from part 1) of Corollary~\ref{maxapprcor} that 
$v(x-K(a_{i-1}))$ has a largest element. However, as $x\in L'\setminus 
K(a_{i-1})$, this is not the largest element $v(x-x)=\infty$ of $v(x-L')$.

\pars
We are going to prove that there is a homogeneous approximation for $x-a_{i-1}$ 
over $K$. By what we have shown above, $v(x-a_{i-1})=\max v(x-K(a_{i-1}))$ is 
not the  largest element of $v(x-L')$, so there is some $z\in L'\setminus
K(a_{i-1})$  such that $v(x-z)>v(x-a_{i-1})$; in the algebraic case we can choose 
$z=x$. We have that $v(z-a_{i-1})=v(x-a_{i-1})$. In all cases, $(K(z,a_{i-1})|K,v)$ is a 
subextension of $(L'|K,v)$, so by part 4) of Lemma~\ref{tameprop}, it is a tame extension. 

We set $b:=z-a_{i-1}$. If e is the smallest natural number such that $\rme vb\in 
vK$, then e is not divisible by $\chara Kv$. Further, if $c\in K$ is such 
that $vcb^{\rm e}=0$, then $cb^{\rm e}v$ is separable-algebraic over $Kv$. Hence the
assumptions of the first part of Lemma~\ref{exkrap} are satisfied and we obtain the 
existence of a homogeneous approximation $a\in \tilde K$ of $b$ over $K$.
Also in the transcendental case, $a$ is a homogeneous approximation of 
$x-a_{i-1}$ over $K$ since 
\[
v(x-a_{i-1}-a)\>\geq\>\min\{v(z-a_{i-1}-a),v(x-z)\}\> >\> 
v(z-a_{i-1})\>=\>v(x-a_{i-1})\>.
\]
We set $\tilde{a}_i:=a_{i-1}+a$. 

Assume that $v(x-K(\tilde{a}_i))$ has a maximum. Then we pick $a_i\in K(\tilde{a}_i)$ such
that $v(x-a_i)=\max v(x-K(\tilde{a}_i))$. From Lemma~\ref{KS4,5lem} we infer that $K(a_i)=
K(\tilde{a}_i)$. 
If $v(x-K(a_i))$ has no maximum, then we set $a_i:=\tilde{a}_i\,$. In both cases, 
(KS2) and (KS3) hold. This shows that ${\eu S}_i:=(a_j)_{1\leq j\leq i}$
a key sequence for $(K(x)|K,v)$. By Lemma~\ref{KSinL} and part 3) of Lemma~\ref{WP2},
using also that $(L',v)$ is henselian since $(K,v)$ is,
\[
K_{{\eu S}_i}\>\subseteq\>K(x)^h\cap\widetilde{K}\>=\>\ic(K(x)|K,v)\>\subseteq\>
\ic(L'(x)|L',v)\>=\>L'\>.
\]
This completes the induction step. 

If the construction stops at some $n\in\N$, then we set ${\eu S}:={\eu S}_n\,$. If the 
construction does not stop (which can only happen in the transcendental case), then we set 
${\eu S}:=\bigcup_{i\in\N} {\eu S}_i\,$; it is straightforward to prove that this is a 
key sequence for $(K(x)|K,v)$. As we have shown above, if the construction stops at some 
$n\in\N$, then (CKS1) or (CKS2) hold, that is, the key sequence ${\eu S}={\eu S}_n$ is
pre-complete, and it is complete if $x\in\widetilde{K}$, in which case (CKS1) must hold. 
If the sequence is not complete, then that means that $x$ is limit of a pseudo Cauchy 
sequence of transcendental type in $(K(a_n),v)$. But this is also a pseudo Cauchy 
sequence of transcendental type in $(L',v)$, as shown in the proof of part 4) of 
Lemma~\ref{WP2}. If $(L'(x)|L',v)$ is valuation transcendental, then such a pseudo Cauchy
sequence cannot exist, showing that ${\eu S}$ must be complete.

Finally, if the construction does not stop, then the index set $S$ of ${\eu S}$ is 
$\N$ and ${\eu S}$ is complete by definition.
\end{proof}

\begin{remark}
It is an open problem whether it can be shown that also $a_i-a_{i-1}$ is a homogeneous 
approximation of $x-a_{i-1}\,$. Likewise, one may be tempted to believe that also 
$\tilde{a}_i$ is homogeneous over $K$, but this appears to be not the case in general.
\end{remark}

\pars
We can give the following characterization of elements in tame extensions:
\begin{corollary}
An element $y\in \widetilde{K}$ belongs to a tame extension of a henselian
field $(K,v)$ if and only if there is a finite key sequence
$(a_i)_{1\leq i\leq k}$ for $(K(y)|K,v)$ such that $y=a_k\,$.
\end{corollary}
\begin{proof}
Suppose that such a sequence exists. Then $y=a_k\in K_{\eu S}$, and by 
Proposition~\ref{hit}, $K_{\eu S}$ is a tame extension of $K$. 

The converse is part of Proposition~\ref{existcks}.
\end{proof}

\begin{corollary}                        
Assume that $(K,v)$ is a henselian field. Then $(K,v)$ is a tame field
if and only if for every transcendental extension $(K(x)|K,v)$
there exists a complete key sequence for $(K(x)|K,v)$.
\end{corollary}
\begin{proof}
The implication ``$\Rightarrow$'' is part of Proposition~\ref{existcks}.
The converse follows from \cite[Proposition 3.12]{BK}.
\end{proof}

\mn
%
%
%
\subsection{Proof of the main theorems}     \label{sectproofs}
\mbox{ }\sn
{\it Proof of Theorem~\ref{MT1gen}:} \ 
Take a henselian valued field $(K,v)$ and a transcendental extension $(K(x)|K,v)$. 
Assume that there is a tame extension $(L'|K,v)$ such that for some extension of $v$
from $K(x)$ to $L'(x)$, the extension $(L'(x)|L',v)$ is weakly pure. 
By Proposition~\ref{existcks} there exists a pre-complete key sequence ${\eu S}=
(a_i)_{i\in S}$ for $(K(x)|K,v)$ with $K_{\eu S}\subseteq L'$. It satisfies (KS4) by
Lemma~\ref{lemKS4}, and (KS5) and (KS6) by Proposition~\ref{KS5,6}.

We take $Q_i$ to be the minimal polynomial of
$a_i$ over $K$. Then by Proposition~\ref{kskp}, $Q_i$ is a key polynomial and $a_i$ is its
unique maximal root. Assertion (ii) follows from (KS1). 
Assertion (iv) follows from (KS4). Assertion (v) follows from Lemma~\ref{kspcs}. 
Assertion (x) follows from part 2) of Lemma~\ref{xx} together with Proposition~\ref{kskp}.

From Theorem~\ref{thIC} we know that $L:=\ic (K(x)|K,v)=K_{\eu S}^h$. Since $(K,v)$ is 
henselian, so is $K_{\eu S}$, thus $L=K_{\eu S}^h=K_{\eu S}\,$. By definition, $K_{\eu S}=
K(a_i\mid i\in S)$. If $S=\{1,\dots,n\}$, then by Lemma~\ref{KSinL}, $K_{\eu S}=
K_{\eu S}^h=K(a_n)^h=K(a_n)$. This proves assertion (ix). Further, by part 3) of
Lemma~\ref{WP2} and the fact that $L'$ is henselian, $L=\ic(K(x)|K,v)\subseteq
\ic(L'(x)|L',v)=L'$.

\pars
Assume that $S=\N$. Then by definition, ${\eu S}$ is a complete key sequence. From 
Proposition~\ref{kskp} it follows that $(Q_i)_{i\in S}$ is a complete sequence
of key polynomials, so we set $\Omega=\emptyset$. Then assertions (i) and (iii) hold,
assertion (vi) follows from (KS3), and assertions (vii) and (viii) hold trivially. From 
Proposition~\ref{SNpure} and the equality $L=K_{\eu S}$ it follows that the extension 
$(L(x)|L,v)$ is immediate and weakly pure. Since $K_{\eu S}\subseteq\widetilde{K}$, this 
implies that the extension $(K(x)|K,v)$ is valuation algebraic.

\pars
Now assume that $S=\{1,\ldots,n\}$. Consequently, as ${\eu S}$ is a pre-complete key 
sequence and $x$ is transcendental over $K$, 
(CKS2) must hold. Assume first that $v(x-a_n)=\max v(x-\widetilde{K})$. Then by 
Lemma~\ref{WP}, the extension $(K(a_n,x)|K(a_n),v)$ is weakly pure and valuation 
transcendental. Since $L=K_{\eu S}=K(a_n)$, we have actually proved that the extension 
$(L(x)|L,v)$ is weakly pure and valuation transcendental. Moreover, we know from 
Proposition~\ref{kskp} that $(Q_i)_{i\in S}$ is a complete sequence of key polynomials.
As before we set $\Omega=\emptyset$, so that assertions (i), (iii) and (viii) hold. Also
assertion (vii) holds since we will have $\Omega\ne\emptyset$ in the remaining case below.
Further, assertion (vi) for $i=n$ holds by our assumption, and for $i<n$ follows from (KS3).

Now assume that $x$ is the limit of a pseudo Cauchy sequence, say $(z_\nu)_{\nu<\lambda}$ 
where $\lambda$ is some limit ordinal, of transcendental type in $(K(a_n),v)$. 
Since $L=K(a_n)$, it follows from Lemma~\ref{limtpcs} that the extension 
$(L(x)|L,v)$ is weakly pure and immediate. We 
set $\Omega=\{\nu\mid\nu<\lambda\}$ and take $Q_\nu$ to be the minimal polynomial of 
$z_\nu$ over $K$. Then assertion (vi) follows from (KS3), and assertion (vii) holds
because $\Omega\ne\emptyset$. For the proof of assertion (viii) we observe that we have set 
$\Omega=\emptyset$ when $S$ is not finite (hence equal to $\N$); hence it remains to show 
that for all $\nu\in\Omega$, $\deg Q_\nu = \deg Q_n=[K(a_n):K]$. Without loss of 
generality we may assume that $v(x-z_\nu)\geq v(x-a_n)$. Then from Lemma~\ref{KS4,5lem} 
we obtain that $K(z_\nu)=K(a_n)$, which completes our proof of assertion (viii).

In order to prove assertion (iii), we only have to address the polynomials $Q_\nu$ for 
$\nu\in\Omega$. As we have already proved in Proposition~\ref{kskp} that $a_n$ is the 
unique maximal root of its minimal polynomial over $K$, and as we have just shown above 
that $a_n$ may be replaced by $z_\nu\,$, we find that the same holds for $z_\nu\,$.

\pars
Now we have to show that also assertion (i) holds. We know already from 
Proposition~\ref{kskp} that each $Q_i$ for $i\in S$ is a key polynomial. To prove the same 
for each $Q_\nu$, assume that $f\in K[X]$ such that $\delta(f)\geq \delta(Q_\nu)=v(x-z_\nu)$.
Take a root $z$ of $f$ such that $v(x-z)=\delta(f)\geq v(x-z_\nu)\geq v(x-a_n)
>v(x-a_{n-1})$. Then by (KS6), $\deg f =[K(z):K]\geq [K(a_n):K]=\deg Q_n=\deg Q_\nu\,$. 
This proves that $Q_\nu$ is a key polynomial.

To show completeness, take any $f(X) \in K[X]$ and a root $z$ of $f$ such that  $v(x-z)=
\delta(f)$. Take $i\in S$ maximal with $\deg Q_i \leq \deg f$. If $i<n$, then as in the 
proof of Proposition~\ref{kskp} it follows that $\delta(f)=v(x-z)\leq v(x-a_i)=
\delta (Q_i)$. Now assume that $i=n$. If $\delta(f)\leq\delta(Q_n)$, then we are done. 
Hence assume otherwise, so that $v(x-z)>v(x-a_n)$. Since in the present case $x$ 
is the limit of the pseudo Cauchy 
sequence $(z_\nu)_{\nu<\lambda}$ of transcendental type in $(K(a_n),v)$ and this 
cannot have a limit in $K(a_n)$, there must be some $\nu<\lambda$ such that 
$v(x-z_\nu)>v(x-z)$, whence $\delta(Q_\nu)>v(x-z)=\delta(f)$. On the other hand, $\deg 
Q_\nu = \deg Q_n\leq \deg f$. This finishes our proof of the completenes and thus of 
assertion (i).

\pars
We turn to the final assertions of Theorem~\ref{MT1}. We have already shown in all cases
that the extension $(L(x)|L,v)$ is weakly pure. The extension $(K(x)|K,v)$ is valuation
algebraic if and only if the extension $(L(x)|L,v)$ is. This happens precisely if the 
extension $(L(x)|L,v)$ is immediate, and this is the case if and only if 
$S=\N$ or $\Omega$ is nonempty by our construction. In the case of $S=\N$, we know from
Proposition~\ref{SNpure} that $(a_i)_{i\in\N}$ is a pseudo Cauchy sequence in $(L,v)$ of
transcendental type with $x$ as its limit.
\qed

\bn
{\it Proof of Theorem~\ref{MT2}:} \ 
Take a henselian valued field $(K,v)$ and a tame algebraic extension $(K(x)|K,v)$. 
Then by Proposition~\ref{existcks} there exists a complete key sequence ${\eu S}=
(a_i)_{i\in S}$ for $(K(x)|K,v)$. It satisfies (KS4) by Lemma~\ref{lemKS4}, and (KS5) 
and (KS6) by Proposition~\ref{KS5,6}. Since $(K(x)|K,v)$ is a finite extension, 
also $S$ must be finite, hence of the form $S=\{1,\ldots,n\}$ for some $n\in\N$. As 
before, we take $Q_i$ to be the minimal polynomial of
$a_i$ over $K$. Then by Proposition~\ref{kskp}, $(Q_i)_{i\in S}$ is a complete sequence 
of key polynomials for $(K(x)|K,v)$, and assertion (iii) holds. Assertion (iv) of
Theorem~\ref{MT2} holds by (KS4). Assertion (ii) follows from (KS1). Assertion (v) follows
from Lemma~\ref{kspcs}. Assertion (vi) for $i<n$ follows from (KS3), and 
for $i=n$ from (CKS1) which also yields assertion (vii). The analogue of assertion (x) of
Theorem~\ref{MT1} is proved as in the proof of Theorem~\ref{MT1gen}.
\qed

\mn
%
%
%
\subsection{Proof of Theorem~\ref{SCS2}}             \label{sectprSCS2}
\mbox{ }\sn
In view of Theorem~\ref{MT2}, we only have to prove Theorem~\ref{SCS2} for infinite 
extensions $L|K$. There is a rather quick way to do this. Since the tame extension is in
particular separable-algebraic, we can employ the proof of \cite[Theorem 3.16]{KTrans} to
construct an immediate extension of $v$ from $L$ to $L(x)$ such that the extension
$(L(x)|L,v)$ is weakly pure and $L=\ic(K(x)|K,v)$. (Note that for this step it is not 
needed that $(L|K,v)$ be a tame extension.) Then we can apply Theorem~\ref{MT1gen} and
Proposition~\ref{SCS1}.

\parm
Alternatively, Theorem~\ref{SCS2} can be proved by a more direct construction 
which applies the methods developed for the proof of Theorem~\ref{MT2} to suitable
subextensions of increasing finite degree. We choose a sequence $\tilde{y}_j\,$, $j\in\N$ 
such that $K(\tilde{y}_j)\subsetneq K(\tilde{y}_{j+1})$ for all $j$, and 
$L=\bigcup_{i\in\N} K(\tilde{y}_j)$. Then by
part 4) of Lemma~\ref{tameprop}, all subextensions $K(\tilde{y}_j)|K,v)$ of $(L|K,v)$ and 
all extensions $K(\tilde{y}_{j+1})|K(\tilde{y}_j),v)$ are tame. 

Now we build key sequences for the extensions $(K(\tilde{y}_j)|K,v)$ that satisfy a
compatibility condition which will allow us to work with the union over these sequences. 
We proceed by induction on $j$. At every step we will adjust the element $\tilde{y}_j\,$,
replacing it by an element $y_j$ such that $K(y_j)=K(\tilde{y}_j)$ and thus in the end,
$L=\bigcup_{j\in\N} K(y_j)$. In fact, at every step we will construct a complete key 
sequence for the extension $(K(y_j)|K,v)$.

We set $y_1=\tilde{y}_1\,$. From Proposition~\ref{existcks} we obtain a complete key 
sequence ${\eu S}_1=(a_i)_{1\leq i\leq n_1}$ for the extension $K(y_1)|K,v)$.
Having already chosen a suitable element $y_j\in L$ with $K(y_j)=K(\tilde{y}_j)$
and constructed a complete key sequence ${\eu S}_j=(a_i)_{1\leq i\leq n_j}$ for the 
extension $(K(y_j)|K,v)$, we note that $y_j=a_{n_j}$ and proceed as follows. First, we 
choose $z_j\in K(y_j)$ such that $v(\tilde{y}_{j+1}-z_j)=\max v(\tilde{y}_{j+1}-K(y_j))$. 
This is possible by part 1) of Corollary~\ref{maxapprcor}. Then we choose $c_j\in K$ such 
that 
\begin{equation} 
vc_j(\tilde{y}_{j+1}-z_j)\> >\> v(y_j-a_{n_j-1})
\end{equation}
and set 
\begin{equation}
y_{j+1}\>:=\>y_j+c_j(\tilde{y}_{j+1}-z_j)\>. 
\end{equation}
At this point we note that since $y_j,z_j,c_j\in K(y_j)=K(\tilde{y}_j)\subset
K(\tilde{y}_{j+1})$, we have that 
\begin{equation}                   \label{aj+1taj+1}
y_{j+1}\,\in\, K(\tilde{y}_{j+1})\;\mbox{ and }\; \tilde{y}_{j+1}\>=\>c_j^{-1}
(y_{j+1}-y_j)+z_j\,\in\, K(y_{j+1},y_j)\>.
\end{equation}
We show that 
\begin{equation}             \label{max}
v(y_{j+1}-y_j)\>=\>\max v(y_{j+1}-K(y_j))\>. 
\end{equation}
If this were not true, there would be $y'_j\in K(y_j)$ such that $v(y_{j+1}-y'_j)>
v(y_{j+1}-y_j)$. By the definition of $y_{j+1}\,$, this is equivalent to
\[
v(y_j+c_j(\tilde{y}_{j+1}-z_j)-y'_j) \> >\> vc_j(\tilde{y}_{j+1}-z_j)\>, 
\]
which in turn is equivalent to
\[
v(\tilde{y}_{j+1}-(z_j+c_j^{-1}(y'_j-y_j))\> >\> v(\tilde{y}_{j+1}-z_j)\>,  
\]
contradicting our choice of $z_j\,$.

\pars 
By construction of $y_{j+1}\,$,
\begin{equation}
v(y_{j+1}-y_j)\>=\> vc_j(\tilde{y}_{j+1}-z_j) \> >\> v(y_j-a_{n_j-1}) \> \geq\> 
v(y_j-a_i) 
\end{equation}
for $i<n_j$. Thus, part 1) of Lemma~\ref{xx} shows that $(a_i)_{i<n_j}$ is a key sequence 
for $(K(y_{j+1})|K,v)$. As $\tilde{a}_{n_j}-a_{n_j-1}$ is a homogeneous approximation of 
$y_j-a_{n_j-1}$ over $K$, we find that
\[
v(y_j-\tilde{a}_{n_j})\>=\> v(y_j-a_{n_j-1}-(\tilde{a}_{n_j}-a_{n_j-1}))\>>\>   
v(y_j-a_{n_j-1})\>.
\]
Further,
\[
v(y_{j+1}-a_{n_j-1})\>=\>\min\{v(y_{j+1}-y_j),v(y_j-a_{n_j-1})\}\>=\>v(y_j-a_{n_j-1})\>. 
\]

Using this, we compute:
\begin{eqnarray*}
v(y_{j+1}-a_{n_j-1}-(\tilde{a}_{n_j}-a_{n_j-1}))&=& v(y_{j+1}-\tilde{a}_{n_j})
\>\geq\> \min\{v(y_{j+1}-y_j),v(y_j-\tilde{a}_{n_j})\}\\
&>& v(y_j-a_{n_j-1})\>=\>v(y_{j+1}-a_{n_j-1}) \>,
\end{eqnarray*}
shows that (KS2) holds for $y_{j+1}$ in place of $x$ and $n_j$ in place of $i$. With the 
same chices for $x$ and $i$, also (KS3) holds by equation (\ref{max}). We have now proved 
that ${\eu S}_j$ is a key sequence for $K(y_{j+1})|K,v)$.

Using the methods of the proof of Proposition~\ref{existcks}, we now extend ${\eu S}_j$
to a complete key sequence ${\eu S}_{j+1}=(a_i)_{1\leq i\leq n_{j+1}}$ for
$K(y_{j+1})|K,v)$. In particular, we have that $y_{j+1}=a_{n_{j+1}}$ and by (KS4),
\[
K(y_j)\>=\>K(a_{n_j})\>\subset\>K(a_{n_{j+1}}) \>=\>K(y_{j+1})\>,
\]
hence by (\ref{aj+1taj+1}), $K(y_{j+1})=K(\tilde{y}_{j+1})$. This completes our induction.
It remains to prove that ${\eu S}:=\bigcup_{j\in\N} {\eu S}_j= (a_i)_{i\in\N}$ gives rise 
to a strongly complete sequence of key polynomials $(Q_i)_{i\in \N}$ for $(L|K,v)$ by
taking each $Q_i$ to be the minimal polynomial of the element $a_i$ over $K$.

In order to show that (SCKP1) holds, we first observe that $L=\bigcup_{j\in\N} K(y_j)=
\bigcup_{j\in\N} K(a_{n_j})\subseteq K(a_i\mid i\in S)$. On the other hand, for every 
$i\in S$ there is $j\in\N$ such that $i\leq n_j\,$, and by Lemma~\ref{KSinL}, $a_i\in 
K(a_{n_j})$. Hence $L=K(a_i\mid i\in S)$.

In order to show that (SCKP2) holds, pick any $k\in\N$. Choose $j\in\N$ such that $k\leq
n_j\,$. By construction, $(a_i)_{1\leq i\leq n_j}$ is a complete key sequence for the 
extension $K(a_{n_j}|K,v)$. Hence by part 2) of Lemma~\ref{xx}, $(a_i)_{i\leq k}$ is a
complete key sequence for $(K(a_k)|K,v)$. Now Proposition~\ref{kskp} shows that 
$(Q_i)_{i\leq k}$ is a complete sequence of key polynomials for $(K(a_k)|K,v)$, and that 
property 3) holds. The validity of assertions (iv), (v), (vi) and (x) is shown as in the 
proof of Theorem~\ref{MT2}. This completes the proof of the theorem.
\qed

\bn
%
%
%
\section{Construction of extensions}     \label{sectconstr}
In this section we show how extensions of the valuation from a tame field $(K,v)$ to a
transcendental extension $K(x)$ can be constructed by use of implicit function fields,
key sequences and pseudo 
Cauchy sequences. At the same time, we introduce a basic classification of these 
extensions. The first criterion is whether we want the implicit function field $L$ to be 
a finite or infinite extension of $K$. Note that as $(K,v)$ is henselian, the extension 
of $v$ to every algebraic extension field of $K$ is unique. Using Theorem~\ref{MT1} and
Corollary~\ref{MT1cor} together with Lemma~\ref{WP}, we obtain the following case 
distinction:
\sn
Case A: \ $L$ is a finite extension of $K$. As $(L(x)|L,v)$ is weakly pure, by
Lemma~\ref{WP} we have three subcases:
\sn
A.1: For some $a\in L$, the value $v(x-a)$ is non-torsion over $vK$.
\sn
A.2: For some $a\in L$ and some $d\in K$ and $e\in\N$, 
$vd(x-a)^{\rm e}=0$ and $d(x-a)^{\rm e}v$ is transcendental over $Kv$.
\sn
A.3: The element $x$ is limit of a pseudo Cauchy sequence in $(L,v)$ of transcendental 
type.
\sn
Case B: \ $L$ is an infinite extension of $K$. For it to be an implicit constant field, 
it must be countably generated over $K$. 

\parm
We will now discuss the construction methods in all cases in detail. Let us first consider
Case A. We take $a\in\widetilde{K}$ such that $L=K(a)$. If we assign to $x-a$ a value in 
some ordered abelian group containing $v\widetilde{K}$ that is non-torsion over $vK$
(i.e., not contained in $v\widetilde{K}$),
then we are in case A.1. If on the other hand we pick $d\in K$ and $e\in\N$ such that 
$vd(x-a)^{\rm e}=0$ and have extended $v$ so that $d(x-a)^{\rm e}v$ becomes transcendental 
over $Kv$, then we are in case A.2. In both cases, for every $b\in \widetilde{K}$ we 
have that 
\[
v(x-b)\>=\>\min\{v(x-a),v(b-a)\}\>. 
\]
Consequently, for each polynomial $f\in \widetilde{K}[X]$, if we write $f(X)=c\prod_{i=1}^n 
(X-b_i)$, then 
\[
vf(x)\>=\> vc+\sum_{i=1}^n \min\{v(x-a), v(b_i-a)\}\>.
\]
This shows that in cases A.1 and A.2 the extension of $v$ from $K$ to $K(x)$ is uniquely
determined by our choice of $a$ and the information on $x-a$ in the respective cases.

\parm
Assume now that we decide for case A.3. We note that every 
pseudo Cauchy sequence in $(K(a),v)$ without a limit in $K(a)$ must be of transcendental 
type by part 2) of Corollary~\ref{maxapprcor}, as $(K(a),v)$ is a tame field by part 1) of
Lemma~\ref{tameprop}. This implies that it uniquely determines an extension of $v$ 
from $K(a)$ to $K(x,a)$, and thus determines an extension of $v$ from $K$ to $K(x)$.
Unless $(K,v)$ and hence also its finite extension $(K(a),v)$ is a maximal field, we may
choose any pseudo Cauchy sequence in $(K(a),v)$ without a limit in $K(a)$ for the 
construction of the extension.

\pars
In all of the above cases, $(L(x)|L,v)$ is weakly pure, and it follows from 
Lemma~\ref{3.7KTrans} that $L$ is the implicit constant field of $L(x)|L,v)$. To 
obtain that $L$ is also the implicit constant field of $K(x)|K,v)$, we 
refine our construction as follows. We write $\alpha:=\Kras(a,K)$. In cases A.1 and A.2
we may choose the value $v(x-a)$ to be larger that $\alpha$. In case A.3, provided that 
$(K,v)$ is not a maximal field, we choose any 
pseudo Cauchy sequence $(a_\nu)_{\nu<\lambda}$ in $(K(a),v)$ without a limit in $K(a)$.
After multiplying all $a_\nu$ with an element $c\in K$ of high enough value, we may 
assume that there is some
$\mu<\lambda$ such that $v(a_{\mu+1}-a_\mu)>\alpha$. Now we set $b_\nu:=a_\nu-a_{\mu}+a$.
Then also $(b_\nu)_{\nu<\lambda}$ is a pseudo Cauchy sequence  in $(K(a),v)$ without a 
limit in $K(a)$. We extend $v$ from $K(a)$ to $K(x,a)$ by taking $x$ to be a limit of this
pseudo Cauchy sequence. Then
\[
v(x-a)\>=\>v(x-b_{\mu})\>=\>v(b_{\mu+1}-b_{\mu}) \>=\>v(a_{\mu+1}-a_{\mu})
\> >\> \alpha\>.
\]
In all cases, from Proposition~\ref{vkras} we obtain that $a\in K(x)^h$. Hence,
\[
L\>=\>K(a)\>\subseteq\>K(x)^h\cap\tilde{K}\>\subseteq\>\ic(K(x)|K,v)\>\subseteq\>
\ic(L(x)|L,v) \>=\> L\>,
\]
which shows that $L=\ic(K(x)|K,v)$.

\parm
Finally, assume that we are picking a countably generated infinite algebraic extension 
$L$ of $K$ to be our implicit constant field. In this case, the proof of 
\cite[Proposition 3.16]{KTrans} yields a pseudo Cauchy sequence of transcendental type
in $(L,v)$ which determines an extension of $v$ from $L$ to $L(x)$ such that $L=
\ic (K(x)|K,v)$.
In contrast to the previous cases, even if $(K,v)$ is a maximal field, $(L,v)$ is not
(see \cite[Theorem 1.1]{BlKu2}), and there will always be the necessary pseudo Cauchy 
sequence in $(L,v)$ for the construction of the extension.

\bn
\bn
\bn

\end{document}